\documentclass[11pt,a4paper,english]{article}

\usepackage[latin1]{inputenc}
\usepackage[centertags]{amsmath}
\usepackage{color}
\usepackage{newlfont}
\usepackage{graphicx}
\usepackage{babel,amsmath,amssymb,amsbsy,amsfonts,latexsym,amsthm,mathrsfs,color}
\usepackage{dsfont}

\DeclareMathOperator*{\essinf}{ess\,inf}
\DeclareMathOperator*{\esssup}{ess\,sup}

\definecolor{mycolorred}{rgb}{1, 0, 0}

\newtheorem{Theorem}{Theorem}[section]
\newtheorem{Lemma}[Theorem]{Lemma}

\newtheorem{Remark}[Theorem]{Remark}
\newtheorem{Assumptions}[Theorem]{Assumption}

\makeatletter  
\@addtoreset{equation}{section}

\makeatother  
\newcommand{\beq}{\begin{equation}}
\newcommand{\eeq}{\end{equation}}
\def\<{\langle}
\def\>{\rangle}

\def\={{\,=\,}}
\def\+{{\,+\,}}
\def\^gamma{{\tilde{\gamma}}}

\begin{document}

\title{\textbf{Generalized Kuhn-Tucker Conditions for\\ N-Firm Stochastic Irreversible Investment \\  under Limited Resources}\footnote{These results extend a portion of the second author Ph.D.\ dissertation \cite{tesimia} under the supervision of the first and the third authors. Financial support by the German Research Foundation (DFG) via grant Ri--1128--4--1 is gratefully acknowledged by the second and third author.}}
\author{Maria B.\ Chiarolla \thanks{Dipartimento di Metodi e Modelli per l'Economia, il Territorio e la Finanza, Universit\`a di Roma `La Sapienza', Via del Castro Laurenziano 9, 00161 Roma, Italy; \texttt{maria.chiarolla@uniroma1.it}}
 \and Giorgio Ferrari\thanks{Corresponding author.\ Center for Mathematical Economics, Bielefeld University, Universitätsstraße 25, D-33615 Bielefeld, Germany; \texttt{giorgio.ferrari@uni-bielefeld.de}}
 \and  Frank Riedel\thanks{Center for Mathematical Economics, Bielefeld University, Universitätsstraße 25, D-33615 Bielefeld, Germany; \texttt{friedel@uni-bielefeld.de}}}
\date{\today}
\maketitle

\textbf{Abstract.} In this paper we study a continuous time, optimal stochastic investment problem under limited resources in a market with $N$ firms.
The investment processes are subject to a time-dependent stochastic constraint.
Rather than using a dynamic programming approach, we exploit the concavity of the profit functional to derive some necessary and sufficient first order conditions for the corresponding \textsl{Social Planner} optimal policy. Our conditions are a stochastic infinite-dimensional generalization of the Kuhn-Tucker Theorem. The Lagrange multiplier takes the form of a nonnegative optional random measure on $[0,T]$ which is flat off the set of times for which the constraint is binding, i.e.\ when all the fuel is spent.
As a subproduct we obtain an enlightening interpretation of the first order conditions for a single firm in Bank \cite{Bank}.
In the infinite-horizon case, with operating profit functions of Cobb-Douglas type, our method allows the explicit calculation of the optimal policy in terms of the `base capacity' process, i.e.\ the unique solution of the Bank and El Karoui representation problem \cite{BankElKaroui}.

\smallskip

{\textbf{Keywords}}:
stochastic irreversible investment, optimal stopping, the Bank and El Karoui Representation Theorem, base capacity, Lagrange multiplier optional measure.

\smallskip

{\textbf{MSC2010 subsject classification}}: 91B70, 93E20, 60G40, 60H25.

\smallskip

{\textbf{JEL classification}}: C02, E22, D92, G31.

\section{Introduction}
\label{Introduction}

In the latest years the theory of irreversible investment under uncertainty has received much attention in Economics as well as in Mathematics (see, for example, the extensive review in Dixit and Pindyck \cite{DixitPindyck}).
From the mathematical point of view, optimal irreversible investment problems under uncertainty are singular stochastic control problems. In fact, the economic constraint that does not allow disinvestment may be modeled as a `monotone follower' problem; that is, a problem in which investment strategies are given by nondecreasing stochastic processes, not necessarily absolutely continuous with respect to the Lebesgue measure as functions of time. The application of `monotone follower' problems to Economics started with the pioneering papers by Karatzas \cite{Karatzas81}, Karatzas and Shreve \cite{KaratzasShreve84}, El Karoui and Karatzas \cite{KaratzasElKarouiSkorohod} (among others). These Authors studied, by probabilistic arguments, the problem of optimally minimizing a convex cost (or optimally maximizing a concave profit) functional when the controlled diffusion is a Brownian motion tracked by a nondecreasing process, i.e.\ the monotone follower. They showed that any such control problem is closely linked to a suitable optimal stopping problem whose value function $v$ is the derivative of the value function $V$ of the original control problem. 

In the last decade several papers handled singular stochastic control problems of the monotone follower type by deriving first order conditions for optimality and without relying on any Markovian or diffusive setting. That is the case, for instance, of Bank and Riedel \cite{BankRiedel1} in which the Authors studied an intertemporal utility maximization problem with Hindy, Huang and Kreps preferences, of Bank and Riedel \cite{BankRiedel2} in which the optimal dynamic choice of durable and perishable goods is analyzed, or of Riedel and Su \cite{RiedelSu} in which a very general irreversible investment problem with unlimited resources is treated. In these papers the optimal consumption, or investment policy, is constructed as the running supremum of a desirable value. Such level of satisfaction is the optional solution of a stochastic backward equation in the spirit of Bank-El Karoui (cf.\ \cite{BankElKaroui}, Theorem $3$) and may be represented in terms of the value functions of a family of standard optimal stopping problems.

The investment problem becomes even harder if one takes into account the fact that the available resources may be limited. The problem turns into a `finite-fuel' singular stochastic control problem since the total amount of effort (fuel) available to the controller (for example, the firm's manager) is limited.
The mathematical literature on this field started in $1967$ with Bather and Chernoff \cite{Bather2} in the context of controlling the motion of a spaceship. Finite fuel monotone follower problems were then studied by Bene\v s, Shepp and Witsenhausen in $1980$ \cite{Benes}. In $1985$ Chow, Menaldi and Robin \cite{ChowMenaldi} and Karatzas \cite{Karatzas85} used a PDE approach and purely probabilistic arguments, respectively, to show that the optimal policy of a `monotone follower' problem with constant finite fuel is `follow the unconstrained optimal policy until there is some fuel to spend'. Much more difficult is the case of finite fuel given by a time-dependent process, either deterministic or stochastic. 

In $2005$ Bank \cite{Bank}, without relying on any Markovian assumption, generalized the optimal policy proposed by Karatzas \cite{Karatzas85} to the case of a stochastic, increasing, adapted finite fuel process $\theta$. The Author characterized the optimal policy of a cost minimization problem as the unique process satisfying some first order conditions for optimality (cf.\ \cite{Bank}, Theorem $2.2$), `\textsl{the optimal control should be exercised only when its impact on future costs is maximal; on the other hand, when the cost functional's subgradient tends to decrease, then all the available fuel must be used}'.
More in detail, if $\mathbb{S}(\nu)$ is the Snell envelope of the total cost functional's subgradient $\nabla\mathcal{C}(\nu)$ (i.e., $\mathbb{S}(\nu)(t):=\essinf_{t \leq \tau \leq T}\mathbb{E}\{\nabla\mathcal{C}(\tau)|\mathcal{F}_t\}$), and $\mathcal{M}(\nu) + A(\nu)$ is its Doob-Meyer decomposition into a uniformly integrable martingale $\mathcal{M}(\nu)$ and a predictable, nondecreasing process $A(\nu)$, then Bank \cite{Bank}, Theorem $2.2$, proved that $\nu_{*}$ is optimal if and only if
\begin{equation}
\label{FOCinBankint}
\begin{split}
& \mbox{(i)}\,\,\,\,\,\nu_{*}\,\,\mbox{is flat off}\,\,\{\nabla\mathcal{C}(\nu_{*})= \mathbb{S}(\nu_{*})\},  \\
& \mbox{(ii)}\,\,\,\,\,A(\nu_{*})\,\,\mbox{is flat off}\,\,\{\nu_{*}= \theta\}.
\end{split}
\end{equation}
Moreover, the Author constructed the optimal control $\nu_{*}$ in terms of the `base capacity' process, i.e.\ a desirable value of capacity. Mathematically such process is the optional solution of a suitable Bank-El Karoui representation problem \cite{BankElKaroui}.

In this paper we generalize Bank's single firm problem to the case of a Social Planner in a market with $N$ firms in which the total investment is bounded by a stochastic, time-dependent, increasing, adapted finite fuel $\theta(t)$; that is, the case $\sum_{i=1}^N \nu^{(i)}(t) \leq \theta (t)$ $\mathbb{P}$-a.s.\ for all $t \in [0,T]$. The Social Planner's objective is to pursue a vector $\underline{\nu}_{*} \in \mathbb{R}^N_{+}$ of efficient irreversible investment processes  that maximize the aggregate expected profit, net of investment costs, i.e.
\begin{eqnarray}
\label{profittoEuropaintro}
\sup_{\nu^{(i)}:\,\sum_{i=1}^N\nu^{(i)} \leq \theta}\,\,\sum_{i=1}^N \mathbb{E}\bigg\{\int_0^T e^{-\delta(t)}\,\,R^{(i)}(X(t),\nu^{(i)}(t))dt - \int_{[0,T)} e^{-\delta(t)} d\nu^{(i)}(t)\,\bigg\}.  
\end{eqnarray}
Here the operating profit function $R^{(i)}$ of firm $i$, $i=1,2,...,N$, depends directly on the cumulative control exercised since we do not allow for dynamics of the productive capacity. As in Kobila \cite{Kobila}, and Riedel and Su \cite{RiedelSu} (among others), the uncertain status of the economy is modeled by an exogeneous economic shock $\{X(t), t \in [0,T]\}$. Although our finite fuel $\theta$ is increasing as in Bank \cite{Bank}, his results cannot be directly applied to each firm since for each $i$ the investment bound $\theta - \sum_{j \neq i}\nu^{(j)}$ is not an increasing process. To overcome this difficulty we develop a new approach based on a stochastic generalization of the classical Kuhn-Tucker method. That is accomplished as follows.
By applying a version of Koml\`{o}s' theorem for optional random measures (cf.\ Kabanov \cite{Kabanov}, Lemma $3.5$) we prove existence and uniqueness of optimal irreversible investment policies.
Then we use the concavity of the profit functional to characterize the optimal Social Planner policy as the unique solution of some stochastic Kuhn-Tucker conditions.
The Lagrange multiplier takes the form of a nonnegative optional random measure on $[0,T]$ which is flat off the set of times for which the constraint is binding, i.e.\ when all the fuel is spent. Hence, as a subproduct we obtain an enlightening interpretation of the first order conditions that Bank \cite{Bank} proved for a single firm optimal investment problem. In fact, we show that measure the $dA(\nu_{*})$ in (\ref{FOCinBankint}) is equal to the Lagrange multiplier of our control problem
$$d\lambda(t) := e^{-\delta t} [R_{y}(X(t),\theta(t))-\delta] \mathds{1}_{\mathcal{A}}(t) dt,$$ 
$\mathcal{A} \subseteq \{t \geq 0: \nu_{*}(\omega,t)=\theta(\omega,t)\}$, and so it inherits all the regularity properties of $d\lambda$.
As expected in optimization under inequality constraints, our Lagrange multiplier $\lambda$ can grow only when the resource constraint is binding. Moreover, as a new result, it is absolutely continuous with respect to the Lebesgue measure.

In the case of constant finite fuel, we consider two classical monotone follower problems for which the optimal policy is known. By Ito's Lemma, we are able to explicitly find the compensator part in the Doob-Meyer decomposition of the profit (cost) functional's supergradient (subgradient) and to identify it with the Lagrange multiplier of the optimal investment problem. We show that $d\lambda$ has the usual interpretation of shadow price and, again, it has a density with respect to the Lebesgue measure.

Finally, when the $N$ firms have operating profit functions of Cobb-Douglas type, with a different parameter for each of them, our generalized stochastic Kuhn-Tucker approach allows for the explicit calculation of the Social Planner optimal investment strategy. Such optimal policy is given in terms of the `base capacity' processes $l^{(i)}$, i.e.\ the unique solutions of suitable Bank-El Karoui representation problems \cite{BankElKaroui}. 
Indeed, we show that the optimal Social Planner investment policy for firm $i$, $i = 1,2,...,N$, behaves like that of a monopolistic firm which has at disposal a fraction $\beta_i$ of the available resources $\theta$; that is,
$$\nu^{(i)}_{*}(t) = \sup_{0 \leq u < t}(l^{(i)}(u) \wedge \beta_i(u) \theta(u)) \vee y^{(i)},\qquad i=1,2,...,N,$$
with $y^{(i)}$ initial capacity value for firm $i$.
In particular, that fraction is given by
$$\beta_i(t) := \frac{l^{(i)}(t)}{\sum_{j=1}^N l^{(j)}(t)},$$
and therefore $\beta_i(t)\theta(t)$ represents a `fair amount' of resources that has to be assigned to firm $i$ according to its desirable value of capacity at time $t$, i.e.\ $l^{(i)}(t)$.
Even in this more complicated multivariate case, we derive the explicit form of the absolutely continuous Lagrange multiplier optional measure.

The paper is organized as follows. In Section \ref{TheModel} we set the model. In Section \ref{KTapproachmonopolies} we introduce the generalized stochastic Kuhn-Tucker conditions for the Social Planner problem.
In Section \ref{Applications} we find the Lagrange multiplier optional measure for some `finite-fuel' problems from the literature (cf.\ Bank \cite{Bank} and Karatzas \cite{Karatzas85}, among others). Finally, in Section \ref{CobbDouglasSocialPlannercase} we explicitly solve an $N$-firm Social Planner optimization problem with Cobb-Douglas operating profits and stochastic, time-dependent `finite-fuel'.


\section{The Model}
\label{TheModel}

We consider a market with $N$ firms on a time horizon $T \leq + \infty$. Let $(\Omega, \mathcal{F}, \left\{\mathcal{F}_t\right\}_{t \in [0,T]}, \mathbb{P})$ be a complete filtered probability space with the filtration $\{\mathcal{F}_t, t \in [0,T]\}$ satisfying the usual conditions. 
The cumulative irreversible investment of firm $i$ up to time $t$, $i=1,2,...,N$, denoted by $\nu^{(i)}(t)$, is an adapted process, nondecreasing, left-continuous, finite a.s.\ s.t.\ $\nu^{(i)}(0)=y^{(i)}> 0$.

The firms are financed entirely by equities but we focus primarily on the irreversibility of investments and do not model precisely the rest of the economy. It is reasonable to assume that the firms cannot invest in natural resources as much as they like. In fact, we assume that the total amount of natural resources available at time $t$ is a finite quantity $\theta(t)$; that is,
\beq
\label{constraint}
\sum_{i=1}^N \nu^{(i)}(t) \leq \theta(t),\,\,\,\,\,\mathbb{P}\mbox{-a.s.},\,\,\,\,\,\mbox{for}\,\,\,\,\,t \in [0,T].
\eeq
The stochastic time-dependent constraint $\{\theta(t), t \in [0,T]\}$ is the cumulative amount of resources extracted up to time $t$. It is a nonnegative and nondecreasing adapted process with left-continuous paths, which starts at time zero from $ \theta(0) = \theta_o > 0$. We assume
\beq
\label{integrabilitavincolo}
\mathbb{E}\{\theta(T)\}<+\infty.
\eeq
We denote by $\mathcal{S}_{\theta}$ the nonempty set of admissible investment plans, i.e.
\begin{eqnarray*}
\mathcal{S}_{\theta}&\hspace{-0.25cm}:=\hspace{-0.25cm}&\{\underline{\nu}:\Omega \times [0,T] \mapsto  \mathbb{R}_{+}^N,\,\,\mbox{nondecreasing},\,\,\mbox{left-continuous},\,\,\mbox{adapted process}\,\,\mbox{s.t.} \nonumber \\ &&
\hspace{0.2cm}\nu^{(i)}(0)=y^{(i)},\,\,\mathbb{P}\mbox{-a.s.},\,\,i=1,2,...,N,\,\,\mbox{and}\,\,\sum_{i=1}^N \nu^{(i)}(t) \leq \theta(t),\,\,\mathbb{P}\mbox{-a.s.}\,\,\forall t \in[0,T]\}.
\end{eqnarray*}

Let $\{X(t), t \in [0,T]\}$ be some exogenous real-valued state variable progressively measurable with respect to $\mathcal{F}_t$. It may be regarded as an economic shock, reflecting the changes in technological ouput, demand and macroeconomic conditions which have direct or indirect effect on the firm's profit. At the moment we do not make any Markovian assumption.

We take the capital good as numeraire, hence we express profits, costs etc.\ in real terms, not nominal ones. Hence the price of a unitary investment is equal to one. 
We take the point of view of a fictitious \textit{Social Planner} aiming to maximize the aggregate expected profit, net of investment costs, $\mathcal{J}_{SP}(\underline{\nu})$ (see equation (\ref{JSP}) below), by allocating efficiently the available resources. We denote by $\delta(t)$ the Social Planner discount factor. $\delta(t)$ is a nonnegative, optional process, bounded uniformly in $(\omega, t) \in \Omega \times [0,T]$.
Assumption (\ref{integrabilitavincolo}) ensures 
\beq
\label{costolimitato}
\mathbb{E}\bigg\{\int_{[0,T)} e^{-\delta(t)}\,d\nu^{(i)}(t)\bigg\} < +\infty,\qquad i=1,2,...,N,
\eeq
i.e.\ the investment plan's expected net present value of firm $i$ is finite.

The operating profit function of firm $i$ is $R^{(i)}: \mathbb{R} \times \mathbb{R}_{+} \mapsto \mathbb{R}_{+}$, $i=1,2,...,N$. At time $t$, when the investment of firm $i$ is $\nu^{(i)}(t)$, $R^{(i)}\left(X(t),\nu^{(i)}(t)\right)$ represents the revenue of firm $i$ under the shock process $X(t)$.
The Social Planner problem is
\beq
\label{problemaEuropa}
V_{SP}:=\sup_{\underline{\nu} \in \mathcal{S}_{\theta}}\,\mathcal{J}_{SP}(\underline{\nu}),
\eeq
where
\begin{eqnarray}
\label{JSP}
\mathcal{J}_{SP}(\underline{\nu}):= \sum_{i=1}^N \mathcal{J}_{i}(\nu^{(i)})
\end{eqnarray}
and, for $i=1,2,...,N$,
\beq
\label{profittoatteso}
\mathcal{J}_{i}(\nu^{(i)})=\mathbb{E}\bigg\{\int_0^T e^{-\delta(t)}\,\,R^{(i)}(X(t),\nu^{(i)}(t))dt -  \int_{[0,T)} e^{-\delta(t)} d\nu^{(i)}(t)\,\bigg\}.
\eeq
Notice that $\mathcal{J}_{i}(\nu^{(i)})$ is the expected total profit, net of investment costs, of firm $i$ when the Social Planner picks $\underline{\nu} \in \mathcal{S}_{\theta}$.

The operating profit functions satisfy the following concavity and regularity assumptions.
\begin{Assumptions}
\label{Assumptionpi} 
\end{Assumptions}
{\em \begin{enumerate}
\item For every $x \in \mathbb{R}$ and $i=1,2,...,N$, the mapping $y \mapsto R^{(i)}(x, y)$ is increasing, strictly concave and with $R^{(i)}(x,0)=0$. Moreover, it has continuous partial derivative $ R^{(i)}_{y}(x,y)$ satisfying the Inada conditions $$\lim_{y\rightarrow 0}R^{(i)}_{y}(x,y)= \infty,\,\,\,\,\,\,\,\,\,\,\lim_{y \rightarrow \infty}R^{(i)}_{y}(x,y)= 0.$$
\item The process $(\omega,t)\mapsto e^{-\delta(\omega,t)}R^{(i)}(X(\omega,t), \theta(\omega,t))$\,\, is \,\, $d\mathbb{P} \otimes dt$-integrable, for $i=1,2,...,N$.
\end{enumerate}}
\noindent Under (\ref{integrabilitavincolo}) and Assumption \ref{Assumptionpi} the net profit $\mathcal{J}_{i}(\nu^{(i)})$  is well defined and finite for all admissible plans.

\section{A Stochastic Kuhn-Tucker Approach}
\label{KTapproachmonopolies}

In this Section we aim to characterize the optimal investment plan by means of a gradient approach.
As in Riedel and Su \cite{RiedelSu}, proof of Theorem $2.6$, by applying a suitable version of Koml\`{o}s' Theorem for optional random measures (cf.\ Kabanov \cite{Kabanov}, Lemma $3.5$) we obtain existence and uniqueness of a solution to problem (\ref{problemaEuropa}). In fact, Koml\`{o}s' Theorem states that if a sequence of random variables $(Z_n)_{n \in \mathbb{N}}$ is bounded from above in expectation, then there exists a subsequence $(Z_{n_k})_{k \in \mathbb{N}}$ which converges in the Ces\`{a}ro sense to some random variable $Z$. In our case the limit provided by Koml\`{o}s' Theorem turns out to be the optimal investment strategy.
\begin{Theorem}
\label{ExistenceandUniqueness}
Under {\rm (\ref{integrabilitavincolo})} and Assumption {\rm \ref{Assumptionpi}}, there exists a unique optimal vector of irreversible investment plans $\underline{\nu}_{*} \in \mathcal{S}_{\theta}$ for problem {\rm (\ref{problemaEuropa})}.
\end{Theorem}

\begin{proof}

Let $\underline{\nu} \in \mathcal{S}_{\theta}$ and denote by $\mathcal{H}$ the space of optional measures on $[0,T]$. Then, the investment strategies $\nu^{(i)}$ may be regarded as elements of $\mathcal{H}$, hence $\mathcal{S}_{\theta}\subset \mathcal{H}^{N}$.

Let $(\underline{\nu}_n)_{n \in \mathbb{N}}$ be a maximizing sequence of investment plans in $\mathcal{S}_{\theta}$, i.e.\ a sequence such that $\displaystyle \lim_{n \rightarrow \infty} \mathcal{J}_{SP}(\underline{\nu}_n) = V_{SP}$. 
By (\ref{integrabilitavincolo}) we have that the sequence $(\mathbb{E}\{\nu^{(i)}_n(T)\})_{n \in \mathbb{N}}$ is bounded for $i=1,2,...,N$; in fact,
$\mathbb{E}\{\nu^{(i)}_n(T)\}\leq \mathbb{E}\left\{\theta(T)\right\}<\infty.$
By a version of Koml\`{o}s' Theorem for optional measures (cf.\ Kabanov \cite{Kabanov}, Lemma $3.5$), there exists a subsequence $(\hat{\underline{\nu}}_n)_{n \in \mathbb{N}}$ that converges weakly a.s.\ in the Ces\`{a}ro sense to some random vector $\underline{\nu}_{*} \in \mathcal{H}^{N}$. That is, for $i=1,2,...,N$, we have, almost surely, 
\beq
\label{CesaroSP}
\hat{I}^{(i)}_n(t) := \frac{1}{n}\sum_{j=0}^{n}\hat{\nu}^{(i)}_j(t) \rightarrow \nu^{(i)}_{*}(t),\quad \mbox{as}\quad n \rightarrow \infty,
\eeq
for every point of continuity of $\nu^{(i)}_{*}$, $i=1,2,...,N$.
Notice that $\hat{\underline{\nu}}_n \in \mathcal{S}_{\theta}$ for all $n$ implies that also the Ces\`{a}ro sequence $\hat{\underline{I}}_n$ belongs to $\mathcal{S}_{\theta}$ due to the convexity of $\mathcal{S}_{\theta}$, hence 
$\sum_{i=1}^N \hat{I}^{(i)}_n(t) \leq \theta(t),$ for $n \in \mathbb{N}$.
It follows that, almost surely,
\beq
\label{weakly3}
\sum_{i=1}^N \nu^{(i)}_{*}(t) \leq \theta(t),
\eeq
which means $\underline{\nu}_{*} \in \mathcal{S}_{\theta}$.

Since $(\nu^{(i)}_n)_{n \in \mathbb{N}}$ is a maximizing sequence so is $(\hat{I}^{(i)}_n)_{n \in \mathbb{N}}$ by concavity of the profit functional. Then, applying Jensen inequality and using Assumption \ref{Assumptionpi}, we have
\beq
\label{stimafinaleSP}
\mathcal{J}_{SP}(\underline{\nu}_{*}) \geq \lim_{n \rightarrow \infty}\, \frac{1}{n}\sum_{j=0}^{n}\mathcal{J}_{SP}(\hat{\underline{\nu}}_n) = V_{SP},
\eeq
by dominated convergence theorem. 
Finally, uniqueness follows from the strict concavity of the Social Planner profit functional.
\end{proof}

We now aim to characterize the Social Planner optimal policy as the unique solution of a set of first order generalized stochastic Kuhn-Tucker conditions.
Notice that the strictly concave functionals $\mathcal{J}_i$, $i=1,2,...,N$, admit the supergradient
\beq
\label{additive3}
\nabla_{y}\mathcal{J}_{i}(\nu^{(i)})(t) :=\mathbb{E}\bigg\{\,\int_{t}^T e^{-\delta(s)}\,R^{(i)}_{y}(X(s),\nu^{(i)}(s)) \,ds\,\Big|\,\mathcal{F}_{t}\,\bigg\}-e^{-\delta(t)}\mathds{1}_{\{t < T\}},
\eeq
for $t \in [0,T]$, in the sense that we have
$$ \mathcal{J}_{i}(\mu^{(i)})- \mathcal{J}_{i}(\nu^{(i)}) \le \langle \nabla_{y}\mathcal{J}_{i}(\nu^{(i)}),\, \mu^{(i)} - \nu^{(i)} \rangle $$ for all admissible investment plans $\mu^{(i)}, \nu^{(i)} \in \mathcal{S}_{\theta}$.

\begin{Remark}
\label{optionalsupergradient}
The quantity $\nabla_{y}\mathcal{J}_i(\nu^{(i)})(t)$, $i=1,2,...,N$, may be interpreted as the marginal expected profit resulting from an additional infinitesimal investment at time $t$ when the investment plan is $\nu^{(i)}$.
Mathematically, $\nabla_{y}\mathcal{J}_i(\nu^{(i)})$ is the Riesz representation of the profit gradient at $\nu^{(i)}$.
More precisely, define $\nabla_{y}\mathcal{J}_i(\nu^{(i)})$ as the optional projection of the product-measurable process
\beq
\label{progrmeas}
\Phi_i(\omega,t):= \int_{t}^T e^{-\delta(\omega,s)}\,R^{(i)}_{y}(X(\omega,s),\nu^{(i)}(\omega,s))\,ds\, - \,e^{-\delta(\omega,t)}\mathds{1}_{\{t < T\}}, 
\eeq
for $\omega \in \Omega$ and $t \in [0,T]$. Hence $\nabla_{y}\mathcal{J}_i(\nu^{(i)})$ is uniquely determined up to $\mathbb{P}$-indistinguishability and it holds
$$\mathbb{E}\bigg\{\,\int_{[0,T)} \nabla_{y}\mathcal{J}_i(\nu^{(i)})(t)d\nu^{(i)}(t)\bigg\} = \mathbb{E}\bigg\{\,\int_{[0,T)}\Phi_i(t) d\nu^{(i)}(t)\bigg\}$$
for all admissible $\nu^{(i)}$ (cf.\ Jacod \cite{Jacod}, Theorem 1.33).
\end{Remark}


\subsection{Generalized Stochastic Kuhn-Tucker Conditions}

Let $\mathcal{B}[0,T]$ denote the Borel $\sigma$-algebra on $[0,T]$. Recall that if $\beta$ is a right-continuous, adapted and nondecreasing process, then the bracket operator 
\beq
\label{scalarproduct}
\langle \alpha, \beta \rangle = \mathbb{E}\bigg\{\,\int_{[0,T)} \alpha(t)\, d\beta(t)\, \bigg\}
\eeq
is well defined (possibly infinite) for all processes $\alpha$ which are nonnegative and $\mathcal{F}_T \otimes \mathcal{B}[0,T]$-measurable. Notice that the bracket is preserved when we pass from $\alpha$ to its optional projection $\alpha^{(o)}$ (cf.\ Jacod \cite{Jacod}, Theorem $1.33$); that is
\beq
\label{scalarproduct2}
\langle \alpha, \beta \rangle = \langle \alpha^{(o)}, \beta \rangle.
\eeq

Since the constraint is $\theta(t) - \sum_{i=1}^N \nu^{(i)}(t) \geq 0$, $\mathbb{P}$-a.s.\ for all $t \in [0,T]$ (cf.\ (\ref{constraint})), we define the \textsl{Lagrangian functional} of problem (\ref{problemaEuropa}) as
\begin{eqnarray}
\label{Lagrangian}
\mathcal{L}^{\theta}(\underline{\nu},\lambda)&\hspace{-0.25cm}=\hspace{-0.25cm}& \mathcal{J}_{SP}(\underline{\nu}) + \langle \theta - \sum_{i=1}^N \nu^{(i)}, \lambda \rangle \nonumber \\
&\hspace{-0.25cm}=\hspace{-0.25cm}&\sum_{i=1}^N \mathbb{E}\bigg\{\int_0^T e^{-\delta(t)}\,\,R^{(i)}(X(t),\nu^{(i)}(t))dt - \int_{[0,T)} e^{-\delta(t)} d\nu^{(i)}(t)\bigg\}  \\
&&\hspace{0.5cm}+\,\mathbb{E}\bigg\{\,\int_{[0,T)} \Big[ \theta(t) - \sum_{i=1}^N \nu^{(i)}(t) \Big]d\lambda(t)\, \bigg\}, \nonumber
\end{eqnarray}
where $d\lambda(\omega, t)$ is a nonnegative optional measure, which may be interpreted as the Lagrange multiplier of Social Planner problem $(\ref{problemaEuropa})$.
By using Fubini's Theorem we write the bracket $\langle \theta - \sum_{i=1}^N \nu^{(i)}, \lambda \rangle$ in a  more convenient form, that is

\begin{eqnarray}
\label{scalarLagrange}
\lefteqn{\langle \theta - \sum_{i=1}^N \nu^{(i)}, \lambda \rangle = \mathbb{E}\bigg\{\,\int_{[0,T)} \big[ \theta(t) - \sum_{i=1}^N \nu^{(i)}(t) \big] d\lambda(t)\, \bigg\}} \nonumber \\
&\hspace{-0.25cm}=\hspace{-0.25cm}&\mathbb{E}\bigg\{\,\int_{[0,T)} \Big[\int_{[0,t)} \big( d\theta(s) - \sum_{i=1}^N d\nu^{(i)}(s) \big)\Big] d\lambda(t) \, \bigg\} +\,K\,\mathbb{E}\bigg\{\,\int_{[0,T)} d\lambda(t) \bigg\}  \nonumber \\
&\hspace{-0.25cm}=\hspace{-0.25cm}&\mathbb{E}\bigg\{\,\int_{[0,T)} \Big[\int_{[t,T)} d\lambda(s)\Big]( d\theta(t) - \sum_{i=1}^N d\nu^{(i)}(t))\, \bigg\} +\,K\,\mathbb{E}\bigg\{\,\int_{[0,T)} d\lambda(t) \bigg\},  \nonumber 
\end{eqnarray}
where $K:=\theta_o - \sum_{i=1}^N y^{(i)} = \theta(0) - \sum_{i=1}^N \nu^{(i)}(0) $.
Hence
\begin{eqnarray}
\label{Lagrangian2}
\lefteqn{\mathcal{L}^{\theta}(\underline{\nu}, \lambda) = \mathcal{J}_{SP}(\underline{\nu}) + \langle \theta - \sum_{i=1}^N \nu^{(i)}, \lambda \rangle }\nonumber \\
&\hspace{-0.25cm}=\hspace{-0.25cm}& \sum_{i=1}^N \mathbb{E}\bigg\{\int_0^T e^{-\delta(t)}\,\,R^{(i)}(X(t),\nu^{(i)}(t))dt - \int_{[0,T)} e^{-\delta(t)} d\nu^{(i)}(t)\bigg\} \nonumber \\
& &+\,\mathbb{E}\bigg\{\,\int_{[0,T)} \Big[\int_{[t,T)} d\lambda(s)\Big] \big( d\theta(t) - \sum_{i=1}^N d\nu^{(i)}(t)  \big)\, \bigg\} +\, K\,\mathbb{E}\bigg\{\,\int_{[0,T)} d\lambda(t) \bigg\}. \nonumber 
\end{eqnarray}

We now obtain stochastic Kuhn-Tucker conditions for optimality with a stochastic Lagrange multiplier process that takes care of our dynamic resource constraint.
A similar approach may be found in Bank and Riedel \cite{BankRiedel1} for an intertemporal utility maximization problem under a static budget constraint, with Hindy, Huang and Kreps preferences. From now on, $\mathcal{T}$ denotes the set of all stopping times $\tau$ with values in $[0,T]$, $\mathbb{P}$-a.s.     
\begin{Theorem}
\label{KTEuropa}
If there exists a nonnegative Lagrange multiplier measure $d\lambda(\omega,t)$ such that $\mathbb{E}\{\,\int_{[0,T)} d\lambda(t)\} < \infty$, and the following generalized stochastic Kuhn-Tucker conditions hold true, for $i=1,2,...,N$, for an admissible investment vector $\underline{\nu}_{*}$
\beq
\label{KT1}
\left\{
\begin{array}{ll}
\displaystyle \nabla_{y}\mathcal{J}_i(\nu^{(i)}_{*})(\tau) \leq \mathbb{E}\bigg\{\,\int_{[\tau,T)} d\lambda(s) \,\Big | \mathcal{F}_{\tau}\bigg\},\,\,\,\,\,\,\,\,\,\mathbb{P}\text{-a.s.},\,\,\forall \tau \in \mathcal{T},\\ \\
\displaystyle \int_{[0,T)} \bigg[\nabla_{y}\mathcal{J}_i(\nu^{(i)}_{*})(t) - \mathbb{E}\bigg\{\,\int_{[t,T)} d\lambda(s)\,\Big | \mathcal{F}_{t}\bigg\}\bigg] d\nu^{(i)}_{*}(t) = 0,\,\,\,\,\,\,\,\,\,\mathbb{P}\text{-a.s.},\\ \\
\displaystyle \mathbb{E}\bigg\{\,\int_{[0,T)}  \Big[\theta(t) - \sum_{i=1}^N \nu^{(i)}_{*}(t) \Big] d\lambda(t)\bigg\} = 0,
\end{array}
\right.
\eeq
then $\underline{\nu}_{*}$ is the unique solution of the Social Planner problem (\ref{problemaEuropa}).
\end{Theorem}
\begin{proof}
Let $\underline{\nu}_{*}$ satisfy the first order Kuhn-Tucker conditions (\ref{KT1}) and let $\underline{\nu}$ be an arbitrary admissible plan.
By concavity of $R^{(i)}(x , \cdot)$, $i=1,2,...,N$, and Fubini's Theorem we have
\begin{eqnarray}
\label{Sufficiency1}
\mathcal{J}_{SP}(\underline{\nu}_{*}) - \mathcal{J}_{SP}(\underline{\nu}) &\hspace{-0.25cm} = \hspace{-0.25cm}& \sum_{i=1}^{N}\mathbb{E}\bigg\{\,\int_0^T e^{- \delta(t)} \Big[R^{(i)}(X(t), \nu^{(i)}_{*}(t)) - R^{(i)}(X(t), \nu^{(i)}(t))\,\Big]dt\, \nonumber \\ 
& &\hspace{3cm} - \int_{[0,T)} e^{-\delta(t)} d(\nu^{(i)}_{*}(t) - \nu^{(i)}(t))\bigg\} \nonumber \\
&\hspace{-0.25cm} \geq \hspace{-0.25cm} &   \sum_{i=1}^N \mathbb{E}\bigg\{\,\int_0^T e^{- \delta t} R^{(i)}_{y}(X(t),\nu^{(i)}_{*}(t))\,(\nu^{(i)}_{*}(t) - \nu^{(i)}(t))\,dt \nonumber  \\
& & \hspace{3cm}- \int_{[0,T)} e^{-\delta(t)}  d(\nu^{(i)}_{*}(t) - \nu^{(i)}(t))\bigg\} \\
&\hspace{-0.25cm} = \hspace{-0.25cm}& \sum_{i=1}^N\mathbb{E}\bigg\{\,\int_{[0,T)} \int_{t}^{T} e^{-\delta(s)}R^{(i)}_{y}(X(s),\nu^{(i)}_{*}(s))\,ds\,d(\nu^{(i)}_{*}(t) - \nu^{(i)}(t))\,  \nonumber  \\
& &\hspace{3cm} - \int_{[0,T)} e^{-\delta(t)} d(\nu^{(i)}_{*}(t) - \nu^{(i)}(t))\,\bigg\}\, \nonumber \\
&\hspace{-0.25cm}  = \hspace{-0.25cm}&\sum_{i=1}^N\mathbb{E}\bigg\{\,\int_{[0,T)} \nabla_{y}\mathcal{J}_{i}(\nu^{(i)}_{*})(t)\,d(\nu^{(i)}_{*}(t) - \nu^{(i)}(t))\,\bigg\}, \nonumber 
\end{eqnarray}
where we have used Remark \ref{optionalsupergradient} for the last equality.
Now (\ref{KT1}) implies
\begin{eqnarray}
\label{Sufficiency2}
\mathcal{J}_{SP}(\underline{\nu}_{*}) - \mathcal{J}_{SP}(\underline{\nu})  &\hspace{-0.25cm} \geq \hspace{-0.25cm}& \sum_{i=1}^N\mathbb{E}\bigg\{\,\int_{[0,T)} \nabla_{y}\mathcal{J}_{i}(\nu^{(i)}_{*})(t)\,d(\nu^{(i)}_{*}(t) - \nu^{(i)}(t))\,\bigg\} \nonumber \\
&\hspace{-0.25cm} \geq \hspace{-0.25cm}& \sum_{i=1}^N \mathbb{E}\bigg\{\,\int_{[0,T)}  \mathbb{E}\bigg\{\,\int_{[t,T)} d\lambda(s) \Big| \mathcal{F}_t\bigg\}\,d(\nu^{(i)}_{*}(t) - \nu^{(i)}(t))\,\bigg\} \\
&\hspace{-0.25cm} =  \hspace{-0.25cm}&  \sum_{i=1}^N \mathbb{E}\bigg\{\,\int_{[0,T)} \Big[\int_{[t,T)} d\lambda(s)\Big] d(\nu^{(i)}_{*}(t) - \nu^{(i)}(t))\bigg\},\nonumber 
\end{eqnarray}
and the nonnegativity of $d\lambda(t)$, the admissibility of $\underline{\nu}$, and another application of Fubini's Theorem give
\begin{eqnarray}
\label{Sufficiency3}
\mathcal{J}_{SP}(\underline{\nu}_{*}) - \mathcal{J}_{SP}(\underline{\nu}) &\hspace{-0.25cm}\geq\hspace{-0.25cm}& \sum_{i=1}^N \mathbb{E}\bigg\{\,\int_{[0,T)} \Big[\int_{[t,T)} d\lambda(s)\Big] d(\nu^{(i)}_{*}(t) - \nu^{(i)}(t))\bigg\}   \nonumber \\
&\hspace{-0.25cm} = \hspace{-0.25cm}&\mathbb{E}\bigg\{\,\int_{[0,T)} \sum_{i=1}^N\Big[\nu^{(i)}_{*}(t) - \nu^{(i)}(t)\Big] d\lambda(t) \bigg\} \nonumber \\
&\hspace{-0.25cm} = \hspace{-0.25cm}& \mathbb{E}\bigg\{\,\int_{[0,T)} \Big[\theta(t) - \sum_{i=1}^N \nu^{(i)}(t)\Big] d\lambda(t) \bigg\} \geq 0, \nonumber 
\end{eqnarray}
where the last line follows from (\ref{KT1}), third condition.
\end{proof}

Conditions (\ref{KT1}) are also necessary for optimality under the assumption that
\beq
\label{assnecessity}
\omega \mapsto \theta(\omega,T)\int_0^T e^{-\delta(t)}R^{(i)}(X(\omega,t),\theta(\omega,T))dt\quad \mbox{is}\quad \mathbb{P}\mbox{-integrable},\quad i=1,2,...,N.
\eeq

\begin{Theorem}
\label{thmnecessity}
Assume (\ref{assnecessity}). If $\underline{\nu}_{*}$ is optimal for the Social Planner problem (\ref{problemaEuropa}), then it satisfies the Kuhn-Tucker conditions (\ref{KT1}) for some nonnegative Lagrange multiplier $d\lambda(\omega,t)$ such that $\mathbb{E}\{\,\int_{[0,T)} d\lambda(t)\} < \infty$.
\end{Theorem}
\begin{proof}
The proof splits into two steps. The arguments resemble those of the finite-dimensional Kuhn-Tucker Theorem. Let $\underline{\nu}_{*}$ be optimal for problem (\ref{problemaEuropa}). \vspace{0.25cm}

\textbf{Step 1.}\quad We show that under (\ref{assnecessity}) the optimal policy $\underline{\nu}_{*}$ solves the linearized problem with finite value
\beq
\label{LP}
\sup_{\underline{\nu} \in \mathcal{S}_{\theta}} \,\sum_{i=1}^N\mathbb{E}\bigg\{\,\int_{[0,T)} \Phi_i^{*}(s) d\nu^{(i)}(s)\bigg\} = \sup_{\underline{\nu} \in \mathcal{S}_{\theta}} \,\sum_{i=1}^N\mathbb{E}\bigg\{\,\int_{[0,T)} \nabla_{y}\mathcal{J}_i(\nu^{(i)}_{*})(s) d\nu^{(i)}(s)\bigg\},
\eeq
by Remark \ref{optionalsupergradient} being $\Phi_i^{*}$, as defined in (\ref{progrmeas}), the product-measurable process associated to $\nabla_{y}\mathcal{J}_i(\nu^{(i)}_{*})$, $i=1,2,...,N$.
In fact, let $\underline{\nu}$ be an admissible plan and fix $\epsilon \in (0,1)$.
For $i=1,2,...,N$, define $\nu^{(i)}_{\epsilon}:= \epsilon \nu^{(i)} + (1-\epsilon) \nu^{(i)}_{*}$, and let $\Phi_i^{\epsilon}$ be the product-measurable process defined in (\ref{progrmeas}) associated to $\nabla_{y}\mathcal{J}_i(\nu^{(i)}_{\epsilon})$. Then $\lim_{\epsilon\rightarrow 0} \nu^{(i)}_{\epsilon}(t)= \nu^{(i)}_{*}(t)$, $\mathbb{P}$-a.s., as well as $\lim_{\epsilon\rightarrow 0} \Phi_i^{\epsilon}(t) = \Phi_i^{*}(t)$, $\mathbb{P}$-a.s., by continuity of $R_{y}^{(i)}$.
Optimality of $\underline{\nu}_{*}$, concavity of $y \mapsto R^{(i)}(X(t),y)$ and Fubini's Theorem imply
\begin{equation}
\label{Necessity1}
0 \geq \frac{1}{\epsilon}\,\Big[ \mathcal{J}_{SP}(\underline{\nu}_{\epsilon}) - \mathcal{J}_{SP}(\underline{\nu}_{*})\Big] \geq \sum_{i=1}^N \mathbb{E}\bigg\{\,\int_{[0,T)} \Phi_i^{\epsilon}(t) d(\nu^{(i)}(t) - \nu^{(i)}_{*}(t))\bigg\},
\end{equation}
since $\epsilon (\nu^{(i)} - \nu^{(i)}_{*}) = \nu^{(i)}_{\epsilon} - \nu^{(i)}_{*}$.

To prove that
\beq
\label{aimFatou}
\sum_{i=1}^N \mathbb{E}\bigg\{\,\int_{[0,T)} \Phi_i^{*}(t) \,d(\nu^{(i)}(t) - \nu^{(i)}_{*}(t))\bigg\} \leq 0
\eeq
it suffices to apply Fatou's Lemma since
\begin{equation*}
\label{Fatou}
\sum_{i=1}^N \mathbb{E}\bigg\{\,\int_{[0,T)} \Phi_i^{*}(t)\, d(\nu^{(i)}(t) - \nu^{(i)}_{*}(t))\bigg\} 
\\ \leq  \liminf_{\epsilon \rightarrow 0} \sum_{i=1}^N \mathbb{E}\bigg\{\,\int_{[0,T)} \Phi_i^{\epsilon}(t)\, d(\nu^{(i)}(t) - \nu^{(i)}_{*}(t))\bigg\} \leq 0.
\end{equation*}
For that, however, we must find $\mathbb{P}$-integrable random variables, $G_i(\omega)$, $i=1,2,...,N$, such that
$$I_i^{\epsilon}:= \int_{[0,T)} \Phi_i^{\epsilon}(t)\, d(\nu^{(i)}(t) - \nu^{(i)}_{*}(t)) \geq G_i,\qquad \mathbb{P}\text{-a.s.}$$
We write $I_i^{\epsilon}$ as
\beq
I_i^{\epsilon}  = \int_{0}^T e^{-\delta(t)} R^{(i)}_{y}(X(t),\nu^{(i)}_{\epsilon}(t))(\nu^{(i)}(t) -\nu^{(i)}_{*}(t)) dt -\int_{[0,T)} e^{-\delta(t)} d(\nu^{(i)}(t) -\nu^{(i)}_{*}(t))
\eeq
by Fubini's Theorem. Then, from concavity of $y \mapsto R^{(i)}(x,y)$ and
\beq
\label{stimeCepsilon}
\nu^{(i)}_{\epsilon}(t)
\left\{
\begin{array}{ll}
\leq \nu^{(i)}(t),\,\,\,\,\,\,\mbox{on}\,\,\{t: \nu^{(i)}(t) - \nu^{(i)}_{*}(t) \geq 0\},
\\ \\ \\
> \nu^{(i)}(t),\,\,\,\,\,\,\mbox{on}\,\,\{t: \nu^{(i)}(t) - \nu^{(i)}_{*}(t) < 0\}.
\end{array}
\right.
\eeq
we obtain $I_i^{\epsilon} = \int_{[0,T)} \Phi^i(t)\,d(\nu^{(i)}(t) - \nu^{(i)}_{*}(t)).$
Hence we define 
\beq
\label{Gi}
G_i:=\int_{[0,T)} \Phi^i(t)\,d(\nu^{(i)}(t) - \nu^{(i)}_{*}(t)), \quad \,\,i=1,2,...,N.
\eeq
Now (\ref{integrabilitavincolo}), Assumption \ref{Assumptionpi} and condition (\ref{assnecessity}), imply the integrability of $G_i$ since
$|G_i| \leq C\theta(T)[1+ \int_{0}^T e^{-\delta(t)}R^{(i)}\left(X(t),\theta(T)\right)dt]$, $\mathbb{P}$-a.s., with $C$ constant. \vspace{0.25cm}

\textbf{Step 2.} \quad We now characterize solutions of the linearized problem (\ref{LP}) by some flat-off conditions, like the second and the third ones of (\ref{KT1}). Define
\beq
\label{lambdatilde}
\Psi(t):= \esssup_{\tau \in [t,T]}\mathbb{E}\Big\{\max_{i \leq N}\nabla_{y}\mathcal{J}_i(\nu^{(i)}_{*})(\tau)\Big|\mathcal{F}_t\Big\}.
\eeq 
Thanks to Assumption \ref{Assumptionpi}, $\Psi$ is a supermartingale of class (D) with unique Dobb-Meyer decomposition into a uniformly integrable martingale $M$ and an increasing, predictable process $\lambda$ with integrable terminal value $\lambda(T)$; that is, $\Psi(t)=M(t) - \lambda(t)$. 
Now, by arguments similar to those in the proof of Bank \cite{Bank}, Lemma $2.5$, we show that every solution $\hat{\underline{\nu}}=(\hat{\nu}^{(1)},...,\hat{\nu}^{(N)})$ of (\ref{LP}) must necessarily satisfy the following conditions
\beq
\label{LPcond1}
\left\{
\begin{array}{ll}
\displaystyle \mathbb{E}\bigg\{\,\int_{[0,T)} \left(\nabla_{y}\mathcal{J}_i(\nu^{(i)}_{*})(s) - \Psi(s)\right)\, d\hat{\nu}^{(i)}(s)\bigg\} = 0, \qquad i=1,2,...,N, \\
\displaystyle \mathbb{E}\bigg\{\,\int_{[0,T)} \Big[\theta(s) - \sum_{i=1}^N \hat{\nu}^{(i)}(s)\Big]\, d\lambda(s)\bigg\} = 0.
\end{array}
\right.
\eeq
Then (\ref{LPcond1}) will also hold for $\underline{\nu}_{*}$ by Step $1$. 
We start by noticing that for any $\underline{\nu} \in \mathcal{S}_{\theta}$ we have
\begin{eqnarray}
\label{necessitydis1}
\lefteqn{\sum_{i=1}^N \mathbb{E}\bigg\{\int_{[0,T)}\nabla_{y}\mathcal{J}_i(\nu^{(i)}_{*})(s)\,d\nu^{(i)}(s)\bigg\} \leq \sum_{i=1}^N \mathbb{E}\bigg\{\int_{[0,T)}\Psi(s)\,d\nu^{(i)}(s)\bigg\}} \nonumber \\
& & = \sum_{i=1}^N \mathbb{E}\bigg\{\int_{[0,T)} \Big(\int_s^T d\lambda(u)\Big) d\nu^{(i)}(s)\bigg\} = \sum_{i=1}^N \mathbb{E}\bigg\{\int_{[0,T)}  \nu^{(i)}(s)\, d\lambda(s) \bigg\} \\
& & \leq \mathbb{E}\bigg\{\int_{[0,T)} \theta(s)\, d\lambda(s) \bigg\}, \nonumber
\end{eqnarray}
by definition (\ref{lambdatilde}).
The first equality follows from $\Psi(t)=\mathbb{E}\{\int_t^T d\lambda(s)|\mathcal{F}_t\}$ since $\Psi(T)=0$, whereas Fubini's Theorem yields the second one. Obviously, if $\underline{\nu}$ satisfies (\ref{LPcond1}), we have equalities in (\ref{necessitydis1}). 
On the other hand, if
\beq
\label{valuefumctionlinear}
\sup_{\underline{\nu} \in \mathcal{S}_{\theta}} \,\sum_{i=1}^N\mathbb{E}\bigg\{\,\int_{[0,T)} \nabla_{y}\mathcal{J}_i(\nu^{(i)}_{*})(s)\, d\nu^{(i)}(s)\bigg\} = \mathbb{E}\bigg\{\int_{[0,T)}\theta(s)\,d\lambda(s)\bigg\},
\eeq
then equalities hold through (\ref{necessitydis1}) and we obtain (\ref{LPcond1}).

It remains to show (\ref{valuefumctionlinear}). For every $i=1,2,...,N$ and $k \in \mathbb{N}$ define the sequence of stopping times
\beq
\label{stoppingtimesvalueLP}
\left\{
\begin{array}{ll}
\tau_0^{(i),k}:= \inf\{t \in [0,T) : \nabla_{y}\mathcal{J}_i(\nu^{(i)}_{*})(t) = \Psi(t)\} \wedge T; \\
\tau_j^{(i),k}:= \inf\{t \in [\tau_{j-1}^{(i),k},T) : \nabla_{y}\mathcal{J}_i(\nu^{(i)}_{*})(t) = \Psi(t),\,\,\,\theta(t) > \theta(\tau_{j-1}^{(i),k}) + \frac{1}{k}\} \wedge T, \quad j \geq 1,
\end{array}
\right.
\eeq
and then set
$$\nu^{(i),k}(t) := \frac{1}{N}\sum_{j=0}^{\infty} \theta(\tau_{j}^{(i),k} +) \mathds{1}_{(\tau_{j}^{(i),k}, \tau_{j+1}^{(i),k}]}(t), \quad i=1,2,...,N.$$
As in \cite{Bank}, proof of Lemma $2.5$, we have
$$\sum_{i=1}^N\mathbb{E}\bigg\{\,\int_{[0,T)} \nabla_{y}\mathcal{J}_i(\nu^{(i)}_{*})(s) d\nu^{(i),k}(s)\bigg\} \geq \mathbb{E}\bigg\{\int_{[0,T)}\theta(s)d\lambda(s)\bigg\} + \frac{1}{k}\mathbb{E}\bigg\{\int_{[0,T)}d\lambda(s)\bigg\},$$
and by letting $k \uparrow \infty$ we obtain (\ref{valuefumctionlinear}).

In conclusion, we have shown that, under (\ref{assnecessity}), the solution of problem (\ref{problemaEuropa}) solves (\ref{LP}) as well. On the other hand, any solution to (\ref{LP}) is characterized by the `flat-off conditions' (\ref{LPcond1}) and this concludes the proof.
\end{proof}

\begin{Remark}
We point out that our stochastic Kuhn-Tucker approach may be generalized to the case of investment processes also bounded from below by a stochastic process. In that case the Lagrangian functional is defined in terms of two Lagrange multipliers, $d\lambda_1(\omega,t)$ and $d\lambda_2(\omega,t)$. 
\end{Remark}

\section{Finding the Lagrange Multiplier for Some Known Models}
\label{Applications}

In this Section, we consider some `finite-fuel' problems from the literature (cf.\ Bank \cite{Bank} and Karatzas \cite{Karatzas85}, among others) for which the form of the optimal investment is known (see (\ref{dynamicoptimalsol}), (\ref{optcostoconvesso}) and (\ref{optCobb}) below). We shall provide the explicit form of the corresponding Lagrange multiplier optional measure $d\lambda$ (see (\ref{dlambdaidentification}), (\ref{LagrangemultBenes}) and (\ref{Cobbcompensator2}) below). It turns out that $d\lambda$ differs from $dA$ at most for its sign, where $A$ is the compensator in the Doob-Meyer decomposition of the profit (cost) functional supergradient's (subgradient's) Snell envelope. In particular, we shall prove that $d\lambda$ (and $dA$) is absolutely continuous with respect to the Lebesgue measure.
In the following examples we assume $\delta(t) = \delta t$, with $\delta > 0$, and $T=+\infty$.


\subsection{The Finite Fuel Monotone Follower of Bank \cite{Bank}}
\label{Banksec}

In the setting of Section \ref{TheModel}, under (\ref{integrabilitavincolo}) and Assumption \ref{Assumptionpi}, we take $N=1$ and $T=+\infty$. We set $\nu:=\nu^{(1)}$, $y:=y^{(1)}$, $R:=R^{(1)}$ and $\mathcal{J}:=\mathcal{J}_1.$ Notice that with $$c(\omega, t, \nu(\omega,t)):= -e^{-\delta t}R(X(\omega,t), \nu(\omega,t)),$$ and instantaneous cost of investment $$k(\omega,t):= -e^{-\delta t},$$ we recover Bank's model \cite{Bank}. Recall that Bank's optimal investment (cf.\ Bank \cite{Bank}, Theorem $3.1$) was given by
\beq
\label{dynamicoptimalsol}
\nu_{*}(t) := \sup_{0 \leq s < t}\left(l(s) \wedge \theta(s)\right) \vee y
\eeq
in terms of the `base capacity' process $l(t)$ (cf.\ Riedel and Su \cite{RiedelSu}, Definition $3.1$) which uniquely solves the stochastic backward equation (cf.\ Bank and El Karoui \cite{BankElKaroui}, Theorem $1$ and Theorem $3$)
\beq
\label{dynamicbackward}
\mathbb{E}\bigg\{\,\int_{\tau}^{\infty} e^{- \delta s} R_{y}(X(s),\sup_{\tau \leq u < s}l(u))\, ds\, \Big| \mathcal{F}_{\tau}\bigg\} = e^{-\delta \tau},\,\,\,\,\,\,\forall \tau \in \mathcal{T}.
\eeq
Easily adapting to our setting arguments as in Bank and K\"uchler \cite{BankKuchler}, proof of Theorem $1$, one can show that $l$ has upper right-continuous sample paths; that is, $l(t) = \limsup_{s \downarrow t}l(s)$.
 
We show the optimality of $\nu_{*}(t)$ by means of our generalized Kuhn-Tucker conditions; as a subproduct we  obtain an enlightening interpretation of the first order conditions stated in Bank \cite{Bank}, Theorem $2.2$, for a single firm optimal investment problem. 

Recall that the supergradient of the net profit functional is the unique optional process given by
\beq
\label{recallgradient}
\nabla_{y}\mathcal{J}(\nu)(t) := \mathbb{E}\bigg\{\,\int_{t}^{\infty} e^{- \delta s}R_{y}\left(X(s),\nu(s)\right)\,ds\, \Big| \mathcal{F}_{t}\bigg\} -  e^{- \delta t}.
\eeq
By Theorem \ref{KTEuropa} an investment plan $\nu_{*}(t)$ is optimal if
\beq
\label{dynamicCOND1}
\nabla_{y}\mathcal{J}(\nu_{*})(\tau) \leq \mathbb{E}\bigg\{\,\int_{\tau}^{\infty}d\lambda(s) \Big| \mathcal{F}_{\tau}\bigg\},\,\,\,\mathbb{P}\mbox{-a.s.},\,\, \tau \in \mathcal{T},
\eeq

\beq
\label{dynamicCOND2}
\int_0^{\infty}\left[\nabla_{y}\mathcal{J}(\nu_{*})(t) - \mathbb{E}\bigg\{\,\int_{t}^{\infty} d\lambda(s) \Big| \mathcal{F}_{t}\bigg\}\right] d\nu_{*}(t) = 0,\,\,\,\mathbb{P}\mbox{-a.s.},
\eeq

\beq
\label{dynamicCOND3}
\nu_{*}(t) \leq \theta(t),\,\,\,\mathbb{P}\mbox{-a.s.},\,\,\forall t \geq 0,
\eeq

\beq
\label{dynamicCOND4}
\mathbb{E}\bigg\{\,\int_0^{\infty}  \left(\theta(t) - \nu_{*}(t)\right) d\lambda(t)\bigg\} = 0,
\eeq
for some nonnegative optional random measure $d\lambda(\omega,t)$ such that $\mathbb{E}\{\int_{0}^{\infty}d\lambda(s)\} < +\infty$. One may easily see from (\ref{dynamicCOND4}) that $d\lambda$ is flat off $\{\nu_{*} = \theta\}$.

\begin{Lemma}
\label{piprimoer}
For almost every $\omega \in \Omega$ one has $\left[R_{y}\left(X(\omega,t),\theta(\omega,t)\right) - \delta\right]\mathds{1}_{\{l(\omega,\cdot) > \theta(\omega, \cdot)\}}(t) \geq 0$. 
\end{Lemma}
\begin{proof}
Take $t \geq 0$ arbitrary but fixed. Then, for any stopping time $\tau_1 \geq t$ a.s., equation (\ref{dynamicbackward}) and the decreasing property of $R_y$ in its second argument imply that
$$e^{-\delta t} \leq \mathbb{E}\bigg\{\,\int_{t}^{\tau_1} e^{- \delta s} R_{y}(X(s),\sup_{t \leq u < s}l(u))\, ds\, \Big| \mathcal{F}_{t}\bigg\} + \mathbb{E}\Big\{e^{-\delta \tau_1}\Big|\mathcal{F}_{t}\Big\} \quad \mbox{a.s.},$$
hence
\begin{equation}
\label{secondinequality}
\mathbb{E}\bigg\{\,\int_{t}^{\tau_1} e^{- \delta s} R_{y}(X(s),l(t))\, ds\, \Big| \mathcal{F}_{t}\bigg\} \geq \mathbb{E}\Big\{e^{-\delta t} - e^{-\delta \tau_1}\Big|\mathcal{F}_{t}\Big\}\quad \mbox{a.s.}
\end{equation}
In particular, for $\epsilon > 0$, define $\tau_1(\epsilon):=\inf\{s \geq t: R_{y}\left(X(s),l(t)\right) > R_{y}\left(X(t),l(t)\right) + \epsilon \}$ (with the usual convention $\inf\{\emptyset\}=+\infty$) to obtain that a.s.
\begin{equation}
\label{stimatauepsilon1}
\mathbb{E}\bigg\{\,\int_{t}^{\tau_1(\epsilon)} e^{- \delta s} R_{y}\left(X(s),l(t)\right)\, ds\, \Big| \mathcal{F}_{t}\bigg\} \leq \frac{1}{\delta}\left(R_{y}\left(X(t),l(t)\right) + \epsilon\right)\mathbb{E}\Big\{\,e^{-\delta t} - e^{-\delta \tau_1(\epsilon)}\Big|\mathcal{F}_{t}\Big\}.
\end{equation}
Now (\ref{stimatauepsilon1}) and (\ref{secondinequality}) with $\tau_1 \equiv \tau_1(\epsilon)$ imply $R_y\left(X(t),l(t)\right) + \epsilon \geq \delta$ a.s.\ for all $\epsilon > 0$. It follows $R_y\left(X(t),l(t)\right)\geq \delta$ a.s., and hence $\left[R_{y}\left(X(t),\theta(t)\right) - \delta\right]\mathds{1}_{\{l(\cdot) \geq \theta(\cdot)\}}(t) \geq 0$ a.s.\ for all $t \geq 0$, by concavity of $y \mapsto R(x,y)$.
\end{proof}

In the next Theorem we prove optimality of $\nu_{*}$ as in (\ref{dynamicoptimalsol}) and, as a new result, we explicitly evaluate the form of the associated Lagrange multiplier measure $d\lambda$. Optimality of $\nu_{*}$ can be shown adapting arguments of Bank \cite{Bank}, proof of Theorem $3.1$. However, we provide here the details for the sake of completeness.
\begin{Theorem}
\label{dynamictheorem}
The process $\nu_{*}(t)$ defined in (\ref{dynamicoptimalsol}) is optimal and the Lagrange multiplier $d\lambda(t)$ is absolutely continuous with respect to the Lebesgue measure.
\end{Theorem}
\begin{proof}
It suffices to check the generalized Kuhn-Tucker conditions (\ref{dynamicCOND1})--(\ref{dynamicCOND4}) for $\nu_{*}(t)$.
Obviously $\nu_{*}$ is admissible and satisfies (\ref{dynamicCOND3}).
Recall that the available resources process $\theta(t)$ is nondecreasing and left-continuous.
To show (\ref{dynamicCOND1}), (\ref{dynamicCOND2}) and (\ref{dynamicCOND4}), first of all fix an arbitrary $\tau \in \mathcal{T}$, define 
\beq
\label{tauthetaBank}
\rho_{\theta}(\tau):=\inf\{s \geq \tau: l(s) > \theta(s+)\},
\eeq
and notice that $\rho_{\theta}(\tau)$ is a point of increase for $\sup_{\tau \leq u < s}l(u)$, $s > \rho_{\theta}(\tau)$.
Following the arguments of Bank \cite{Bank}, proof of Theorem $3.1$, we can now evaluate the Snell envelope of $\nabla_{y}\mathcal{J}(\nu_{*})$. 
From (\ref{dynamicoptimalsol}) we have
\begin{eqnarray}
\label{checkCOND12Bank0}
\lefteqn{\mathbb{E}\bigg\{\,\int_{\tau}^{\infty} e^{-\delta s} R_{y}\left(X(s),\nu_{*}(s)\right) ds\, \Big| \mathcal{F}_{\tau}\bigg\} } \nonumber \\
& & = \mathbb{E}\bigg\{\,\int_{\tau}^{\rho_{\theta}(\tau)} e^{-\delta s} R_{y}(X(s), \nu_{*}(s)) ds\,\Big | \mathcal{F}_{\tau} \bigg\}   \nonumber \\
& &\hspace{0.9cm} + \mathbb{E}\bigg\{\int_{\rho_{\theta}(\tau)}^{\infty} e^{-\delta s} R_{y}(X(s),\nu_{*}(s)) ds\, \Big| \mathcal{F}_{\tau}\bigg\}  \\
& & \leq  \mathbb{E}\bigg\{\,\int_{\tau}^{\rho_{\theta}(\tau)} e^{-\delta s} R_{y}(X(s),\sup_{\tau \leq u < s}l(u)) ds\,\Big | \mathcal{F}_{\tau}\bigg\} \nonumber  \\
& &\hspace{0.9cm} + \mathbb{E}\bigg\{\int_{\rho_{\theta}(\tau)}^{\infty} e^{-\delta s} R_{y}(X(s),\sup_{0 \leq u < s}(l(u) \wedge \theta(u)) ds\, \Big| \mathcal{F}_{\tau} \bigg\} \nonumber,  
\end{eqnarray}
since $\nu_{*}(s) \geq \sup_{\tau \leq u < s}l(u)$, for $s \in (\tau,\rho_{\theta}(\tau)]$, and $\sup_{0 \leq u < s}(l(u) \wedge \theta(u)) \geq \sup_{0 \leq u \leq \rho_{\theta}(\tau)}(l(u) \wedge \theta(u)) \geq \theta(\rho_{\theta}(\tau)) \geq  y$, for $s > \rho_{\theta}(\tau)$, by definition of $\rho_{\theta}(\tau)$ and upper right-continuity of $l$. Also, $\sup_{\tau \leq u < s}l(u) = \sup_{\rho_{\theta}(\tau) \leq u < s}l(u)$, for $s > \rho_{\theta}(\tau)$, and (\ref{dynamicbackward}) imply
\begin{eqnarray}
\label{checkCOND12Bank}
\lefteqn{\mathbb{E}\bigg\{\,\int_{\tau}^{\infty} e^{-\delta s} R_{y}\left(X(s),\nu_{*}(s)\right) ds\, \Big| \mathcal{F}_{\tau}\bigg\} } \nonumber \\
& & \leq \mathbb{E}\bigg\{\,\int_{\tau}^{\infty} e^{-\delta s} R_{y}(X(s),\sup_{\tau \leq u < s}l(u)) ds\,\Big | \mathcal{F}_{\tau}\bigg\} \nonumber  \\
& & \hspace{0.9cm} - \mathbb{E}\bigg\{\,\int_{\rho_{\theta}(\tau)}^{\infty} e^{-\delta s} R_{y}(X(s),\sup_{\tau \leq u < s}l(u)) ds\,\Big | \mathcal{F}_{\tau}\bigg\}  \\
& &\hspace{0.9cm} + \mathbb{E}\bigg\{\int_{\rho_{\theta}(\tau)}^{\infty} e^{-\delta s} R_{y}(X(s),\sup_{0 \leq u < s}(l(u) \wedge \theta(u)) ds\, \Big| \mathcal{F}_{\tau} \bigg\} \nonumber \\
& & = e^{-\delta \tau} + \mathbb{E}\bigg\{\int_{\rho_{\theta}(\tau)}^{\infty} e^{-\delta s} \Big[R_{y}(X(s),\sup_{0 \leq u < s}(l(u) \wedge \theta(u)) - \delta\Big] ds\, \Big| \mathcal{F}_{\tau} \bigg\}, \nonumber
\end{eqnarray}
with equality in (\ref{checkCOND12Bank0}) and (\ref{checkCOND12Bank}) if and only if $\tau$ is a point of increase for $\nu_{*}$ (that is, $d\nu_{*}(\tau) > 0$). 

Notice that the last term in the right-hand side of (\ref{checkCOND12Bank}) does coincide with the Snell envelope, $\mathbb{S}(\nu_{*})$, of $\nabla_{y}\mathcal{J}(\nu_{*})$ (cf.\ Bank \cite{Bank}, proof of Theorem $3.1$); that is,
\begin{eqnarray}
\label{SnellenvelopediBank}
& & \mathbb{S}(\nu_{*})(\tau):= \esssup_{\tau \leq \rho \leq +\infty} \mathbb{E}\left\{\nabla_{y}\mathcal{J}(\nu_*)(\rho)|\mathcal{F}_{\tau}\right\} \nonumber \\
& & = \mathbb{E}\bigg\{\int_{\rho_{\theta}(\tau)}^{\infty} e^{-\delta s} \Big[R_{y}(X(s),\sup_{0 \leq u < s}l(u) \wedge \theta(u)) - \delta\Big] ds\, \Big| \mathcal{F}_{\tau} \bigg\}.
\end{eqnarray}
Thanks to Assumption \ref{Assumptionpi}, $\mathbb{S}(\nu_{*})$ is a process of class (D). Indeed, for any stopping time $\tau \in \mathcal{T}$ and for some positive constant $C$ (depending on $y$), one has
\begin{eqnarray}
\label{classD}
|\mathbb{S}(\nu_{*})(\tau)| & \hspace{-0.25cm} \leq \hspace{-0.25cm}& \mathbb{E}\bigg\{\int_{0}^{\infty} e^{-\delta s} \big|R_{y}(X(s), y \vee \sup_{0 \leq u < s}l(u) \wedge \theta(u)) - \delta \big| ds\, \Big| \mathcal{F}_{\tau} \bigg\} \nonumber \\
& \hspace{-0.25cm} \leq \hspace{-0.25cm} & 1 + \frac{1}{y}\mathbb{E}\bigg\{\int_{0}^{\infty} e^{-\delta s} R(X(s),\theta(s))\,ds\Big| \mathcal{F}_{\tau} \bigg\}  \\
& \hspace{-0.25cm} \leq \hspace{-0.25cm} & C\bigg[ 1 + \mathbb{E}\bigg\{\int_{0}^{\infty} e^{-\delta s} R(X(s),\theta(s))\,ds\Big| \mathcal{F}_{\tau} \bigg\}\bigg], \nonumber 
\end{eqnarray}
where for the second inequality we have used that $R(x,\cdot)$ is strictly concave, increasing and such that $R(x,0)=0$, $x \in \mathbb{R}$. 
Also, the process $\{\mathbb{E}\{\int_{0}^{\infty} e^{-\delta s} R(X(s),\theta(s))\,ds | \mathcal{F}_{t}\}\}_{t \geq 0}$ is a uniformly integrable martingale by the second part of Assumption \ref{Assumptionpi} (see, e.g., Revuz and Yor \cite{RevuzYor}, Chapter II, Theorem $3.1$) and hence the family of random variables $\{\mathbb{E}\{\int_{0}^{\infty} e^{-\delta s} R(X(s),\theta(s))\,ds| \mathcal{F}_{\tau}\},\,\tau \in \mathcal{T}\}$ is uniformly integrable by Revuz and Yor \cite{RevuzYor}, Chapter II, Theorem $3.2$. It follows that $\mathbb{S}(\nu_{*})$ is of class (D). 

If $A(\nu_{*})$ denotes the unique predictable (hence optional) increasing process in the Doob-Meyer decomposition of the supermartingale $\mathbb{S}(\nu_{*})$, then $\mathbb{S}(\nu_{*})(t) = \mathbb{E}\left\{ A(\infty) - A(t)\, |\mathcal{F}_t\right\}$, since $\nabla_{y}\mathcal{J}(\nu_{*})(\infty)=0$.
Therefore $\nu_{*}$ (as defined in (\ref{dynamicoptimalsol})) fulfills (\ref{dynamicCOND1}), (\ref{dynamicCOND2}) and (\ref{dynamicCOND4}) if we identify $d\lambda$ with the optional nonnegative random measure $dA(\nu_{*})$; i.e., if we set $d\lambda \equiv dA(\nu_{*})$.

We now aim to find the form of the Lagrange multiplier, that is the compensator part $dA(\nu_{*})$ of the Doob-Meyer decomposition of $\mathbb{S}(\nu_{*})$, and to show that $d\lambda$ is flat off $\{\nu_{*}=\theta\}$.
Recalling (\ref{tauthetaBank}) it is not hard to see that $\{\mathbb{S}(\nu_{*})(u \wedge \rho_{\theta}(t))\}_{u \geq t}$ is an $\mathcal{F}_u$-martingale, for any $t \geq 0$ arbitrary but fixed; thus, $\mathbb{S}(\nu_{*})$ is a martingale until the base capacity $l$ is below the finite fuel $\theta$. 
In fact, by iterated conditioning and the fact that $\rho_{\theta}(u_1 \wedge \rho_{\theta}(t))= \rho_{\theta}(u_2 \wedge \rho_{\theta}(t))$ for all $u_1$, $u_2 \in [t, \infty)$, it follows that $\{\mathbb{S}(\nu_{*})(u \wedge \rho_{\theta}(t))\}_{u \geq t}$ is an $\mathcal{F}_{u \wedge \rho_{\theta}(t)}$-martingale. Then $\{\mathbb{S}(\nu_{*})(u \wedge \rho_{\theta}(t))\}_{u \geq t}$ is an $\mathcal{F}_u$-martingale by Revuz and Yor \cite{RevuzYor}, Corollary $3.6$.

Next, we define
$$\sigma_{\theta}(t):=\inf\{s > t: l(s)\leq \theta(s+)\}$$ 
and we show that $\{\mathbb{S}(\nu_{*})(u \wedge \sigma_{\theta}(t))\}_{u \geq t}$ is an $\mathcal{F}_u$-submartingale. 
In fact, from $\rho_{\theta}(u \wedge \sigma_{\theta}(t)) = u \wedge \sigma_{\theta}(t)$ for all $u \geq t$, it follows that the process
\beq
\label{supermartingalaS}
\bigg\{\mathbb{S}(\nu_{*})(u \wedge \sigma_{\theta}(t)) + \int_t^{u \wedge \sigma_{\theta}(t)} e^{-\delta s} \Big[R_{y}(X(s),\sup_{0 \leq u < s}l(u) \wedge \theta(u)) - \delta\Big]\,ds\bigg\}_{u \geq t}
\eeq
is an $\mathcal{F}_{u \wedge \sigma_{\theta}(t)}$-martingale, and therefore an $\mathcal{F}_u$-martingale.
Moreover, since $\theta(s) \geq \sup_{0 \leq u < s}(l(u) \wedge \theta(u)) \geq \sup_{t \leq u < s}(l(u) \wedge \theta(u)) = \theta(s)$, for any $s \in (t,\sigma_{\theta}(t))$, the integrand in (\ref{supermartingalaS}) is equal to $e^{-\delta s} [R_{y}(X(s),\theta(s)) - \delta]$ and it is nonnegative by Lemma \ref{piprimoer}. Therefore, by uniqueness of the Doob-Meyer decomposition, we may conclude that 
\beq
\label{dlambdaidentification0}
dA(\nu_*)(s)= e^{-\delta s} \Big[R_{y}(X(s),\theta(s)) - \delta\Big]\,ds,\quad s \in [t,\sigma_{\theta}(t)),\,\,t \geq 0.
\eeq 
It follows that $A(\nu_*)$ increases only on the set $\{s \geq 0: l(\omega, s) > \theta(\omega, s+)\}$ (a subset of $\{s \geq 0: \nu_{*}(\omega, s) = \theta(\omega, s)\}$), i.e.\ at those times $s$ such that $s=\rho_{\theta}(s)$ a.s., and hence we may write the Lagrange multiplier $d\lambda = dA(\nu_*)$ as
\beq
\label{dlambdaidentification}
d\lambda(s) = e^{-\delta s} \Big[R_{y}(X(s),\theta(s)) - \delta\Big]\mathds{1}_{\{l(\cdot) > \theta(\cdot\, +)\}}(s)\,ds.
\eeq 

In conclusion, given $\nu_{*}$ as in (\ref{dynamicoptimalsol}), we have shown $\mathbb{E}\{\,\int_{\tau}^{\infty} e^{-\delta s} R_{y}\left(X(s),\nu_{*}(s)\right) ds\,| \mathcal{F}_{\tau}\}$ $\leq e^{-\delta{\tau}} + \mathbb{E}\{\,\int_{\tau}^{\infty} d\lambda(s)| \mathcal{F}_{\tau}\}$, with $d\lambda = dA(\nu_{*})$, and with equality if and only if $\tau$ is a time of increase for $\nu^{*}(\cdot)$. It follows that (\ref{dynamicCOND1})--(\ref{dynamicCOND4}) hold and hence the process (\ref{dynamicoptimalsol}) is optimal by Theorem \ref{KTEuropa}. Moreover, we have proved that the paths of $A(\nu_{*})(t)$ (and hence of $d\lambda$) are absolutely continuous with respect to the Lebesgue measure and have Radon-Nykodym derivative $e^{-\delta t} [R_{y}(X(t),\theta(t))-\delta] \mathds{1}_{\{l(\cdot)>\theta(\cdot\, +)\}}(t)$.
\end{proof}

The argument of the proof of Theorem \ref{dynamictheorem} allows an enlightening interpretation of the first order conditions in Bank \cite{Bank}.
Recall that $\mathbb{S}(\nu)$ is the Snell envelope of the supergradient $\nabla_{y}\mathcal{J}(\nu)$, i.e.
\beq
\label{Snell}
\mathbb{S}(\nu)(t) = \esssup_{t \leq \tau \leq +\infty} \mathbb{E}\left\{\nabla_{y}\mathcal{J}(\nu)(\tau)|\mathcal{F}_t\right\}.
\eeq
Then Bank \cite{Bank}, Theorem $2.2$, shows that the optimal investment plan $\nu_{*}$ is characterized by the following conditions
\beq
\label{FOCinBank}
\left\{
\begin{array}{ll}
\nu_{*}\,\,\mbox{is flat off}\,\,\{\nabla_{y}\mathcal{J}(\nu_{*})= \mathbb{S}(\nu_{*})\}  \\
A(\nu_{*})\,\,\mbox{is flat off}\,\,\{\nu_{*}= \theta\},
\end{array}
\right.
\eeq
where $A(\nu_{*})$ is the predictable increasing process in the Doob-Meyer decomposition of the supermartingale $\mathbb{S}(\nu_{*})$.
Here we have shown that $d\lambda \equiv dA(\nu_{*})$ and therefore that the second first order condition of (\ref{FOCinBank}) coincides with the Kuhn-Tucker condition (\ref{dynamicCOND4}); that is, the Lagrange multiplier can grow only when the constraint is binding. That generalizes in our stochastic, infinite dimensional setting the usual rule one has in optimization under inequality constraint.

\begin{Remark}
The usual interpretation of the Lagrange multiplier as the shadow price of the value function may be heuristically shown as follows. After an integration by parts on the cost term, we may write the value function as
$$V(\theta)=\mathbb{E}\left\{\,\int_{0}^{\infty} e^{-\delta t}\Big[R(X(t),y \vee \sup_{0\leq s < t}(l(s) \wedge \theta(s))) -\delta (y \vee \sup_{0\leq s < t}(l(s) \wedge \theta(s)))\Big]dt\right\}.$$ Now, if $\nu_{*}(t)= y \vee \sup_{0\leq s < t}(l(s) \wedge \theta(s))$, then the derivative (in some sense) of $\nu_{*}$ with respect to $\theta$ does vanish off $\{\nu_{*}=\theta\}$. We thus expect that the `derivative' of $V$ with respect to the constraint $\theta$ is $e^{-\delta t}[R_{y}(X(t),\theta(t)) -\delta]$ once the constraint is binding, which has exactly the form of the density of our Lagrange multiplier (\ref{dlambdaidentification}). From an economic point of view, $e^{-\delta t}[R_{y}(X(t),\theta(t)) -\delta]$ is the marginal profit, net of the user cost of capital, one has at time $t$ if the productive capacity is $\theta(t)$. Therefore, it is exactly the willingness to pay at time $t$ for relaxing the resources constraint.
\end{Remark}


\subsection{Constant Finite Fuel and Quadratic Cost}
\label{ExBenes}

Here we consider a monotone follower problem with constant finite fuel similar to those studied by Karatzas (\cite{Karatzas81}, \cite{Karatzas85}), and Karatzas and Shreve \cite{KaratzasShreve84} (among others). 
In particular, we discuss the example (cf.\ Bank \cite{Bank}, Section $4.1$) of optimal cost minimization for a firm that does not incur into investment's costs and has a running cost flow given by the convex function $c(x,y)=\frac{1}{2}(x-y)^2$ of the economic shock $x$ and the investment $y$. That is, we study the constrained convex minimization problem
\beq
\label{costoconvesso}
\inf_{\nu \in \mathcal{S}_{\theta_o}}\mathcal{C}(\nu) := \inf_{\nu \in \mathcal{S}_{\theta_o}}\mathbb{E}\bigg\{\,\int_0^{\infty} \delta e^{-\delta s}\frac{1}{2}(W(t) - \nu(t))^2 \,dt \bigg\}
\eeq
where $W(t)$ is a standard Brownian motion and $\theta_o$ is the positive constant finite fuel such that $\nu(t) \leq \theta_o$, $\mathbb{P}$-a.s.\ for all $t\geq 0$.
From Bank \cite{Bank} we know that the optimal investment policy is 
\beq
\label{optcostoconvesso}
\nu_{*}(t) = \sup_{0 \leq s < t}\left((W(s) - c) \wedge \theta_o\right) \vee \nu(0),
\eeq
which is the well known strategy of reflecting the Brownian motion at some threshold $c$ until all the fuel is spent (cf.\ Karatzas \cite{Karatzas85}). To identify $c$ notice that 
\beq
\label{gradientecostoconvesso}
\nabla_{y}\mathcal{C}(\nu)(t) = \mathbb{E}\bigg\{\,\int_t^{\infty} \delta e^{-\delta s}(\nu(s) - W(s)) \,ds\,\Big| \mathcal{F}_t \bigg\},
\eeq
and that the backward equation
\beq
\label{backBenes}
\mathbb{E}\bigg\{\,\int_{\tau}^{\infty} \delta e^{-\delta s}\sup_{\tau \leq u < s} l(u) \,ds\,\Big| \mathcal{F}_{\tau} \bigg\} = e^{-\delta \tau}W(\tau),\,\,\,\,\,\forall \tau \in \mathcal{T},
\eeq
is uniquely solved by the base capacity
\beq
\label{backBenes2}
l(s) = W(s) - c,
\eeq
where $c$ is the positive constant $c:=\mathbb{E}\{\,\int_0^{\infty}\delta e^{-\delta s} \sup_{0\leq u < s} W(u) \,ds\},$ by independence and time-homogeneity of Brownian increments.

For this problem we expect to find a nonpositive Lagrange multiplier.
We write the subgradient (\ref{gradientecostoconvesso}) at $\nu_{*}$ as
\begin{eqnarray*}
\label{gradientecostoconvesso2}
\nabla_{y}\mathcal{C}(\nu_{*})(t)&\hspace{-0.25cm} = \hspace{-0.25cm}& \mathbb{E}\bigg\{\,\int_t^{\infty} \delta e^{-\delta s}\left(\nu_{*}(s) - W(s)\right) \,ds\,\Big| \mathcal{F}_t \bigg\} - 0 
\end{eqnarray*}
and we use (\ref{backBenes}) with $l$ given by (\ref{backBenes2}) to have
\beq
\nabla_{y}\mathcal{C}(\nu_{*})(t) = \mathbb{E}\bigg\{\,\int_t^{\infty} \delta e^{-\delta s}\Big[\nu_{*}(s) - \sup_{t \leq u < s}(W(u) - c)\Big] \,ds\,\Big| \mathcal{F}_t \bigg\}.
\eeq
This trivial trick puts us in the same setting as Bank \cite{Bank}, proof of Theorem 3.1 (see also the first part of the proof of our Theorem \ref{dynamictheorem} above). Hence we have that the Snell envelope of the subgradient evaluated at the optimum $\nu_{*}$ (cf.\ (\ref{optcostoconvesso})) is 
\beq
\label{stimaSnellcostoconvesso}
\mathbb{S}(\nu_{*})(t) = \mathbb{E}\bigg\{\,\int_{\rho_{\theta_o}(t)}^{\infty} \delta e^{-\delta s}\Big[\theta_o - \sup_{t \leq u < s}(W(u) - c)\Big] \,ds\,\Big| \mathcal{F}_t \bigg\} \nonumber
\eeq
with 
\beq
\label{taucostoconvesso}
\rho_{\theta_o}(t):=\inf\{ s \geq t : W(u) - c > \theta_o \},
\eeq
by means of (\ref{optcostoconvesso}).
Notice that $\rho_{\theta_o}(t)$ is a time of increase for $\sup_{t \leq u < s}(W(u)-c)$, $s > \rho_{\theta_o}(t)$. Hence we have $\sup_{t \leq u < s}(W(u) - c)=\sup_{\rho_{\theta_o}(t) \leq u < s}(W(u) - c)$ for $s \in (\rho_{\theta_o}(t),+\infty]$. Therefore (\ref{backBenes}) implies
\beq
\label{stimaSnellcostoconvessobis}
\mathbb{S}(\nu_{*})(t)= \mathbb{E}\bigg\{\,\int_{\rho_{\theta_o}(t)}^{\infty} \delta e^{-\delta s}\,\theta_o\, ds\,\Big| \mathcal{F}_t \bigg \} - \mathbb{E}\Big\{\,e^{-\delta \rho_{\theta_o}(t)}W(\rho_{\theta_o}(t))\,\Big| \mathcal{F}_t\Big\}; \nonumber
\eeq
that is,
\beq
\label{stimaSnellcostoconvesso2}
\mathbb{S}(\nu_{*})(t)= \mathbb{E}\Big\{\,e^{- \delta \rho_{\theta_o}(t)} \Big[\theta_o - W(\rho_{\theta_o}(t))\Big]\,\Big| \mathcal{F}_t \Big\}.
\eeq

Now we identify the Lagrange multiplier of problem (\ref{costoconvesso}), that is the compensator part $dA$ of the Doob-Meyer decomposition of $\mathbb{S}(\nu_{*})$.
Given an arbirary $t \geq 0$, we start by showing that $\mathbb{S}(\nu_{*})$ is a martingale until $W - c$ is below the finite fuel $\theta_o$; that is, $\{\mathbb{S}(\nu_{*})(u \wedge \tau_{\theta_0}(t))\}_{u \geq t}$ is an $\mathcal{F}_u$-martingale.
In fact, by Revuz and Yor \cite{RevuzYor}, Corollary $3.6$, it suffices to prove that $\{\mathbb{S}(\nu_{*})(u \wedge \tau_{\theta_0}(t))\}_{u \geq t}$ is an $\mathcal{F}_{u \wedge \rho_{\theta_o}(t)}$-martingale, and that follows by iterated conditioning and the fact that $\rho_{\theta_o}(u_1 \wedge \rho_{\theta_o}(t))= \rho_{\theta_o}(u_2 \wedge \rho_{\theta_o}(t))$ for all $u_1$, $u_2 \in [t, \infty)$. 
Next, we define 
$$\sigma_{\theta_o}(t):=\inf\{s > t: W(s)\leq \theta_o + c\},$$ and we show that $\mathbb{S}(\nu_{*})$ is a submartingale when the base capacity $l$ is above the finite fuel $\theta_o$ (this is a subset of $\{\nu_{*}=\theta_o\}$); that is $\{\mathbb{S}(\nu_{*})(u \wedge \sigma_{\theta_o}(t))\}_{u \geq t}$ is an $\mathcal{F}_u$-submartingale. Again, as above, it suffices to prove that it is an an $\mathcal{F}_{u \wedge \sigma_{\theta_o}(t)}$-submartingale. In fact, from $\rho_{\theta_o}(u \wedge \sigma_{\theta_o}(t)) = u \wedge \sigma_{\theta_o}(t)$ for all $u \geq t$, the martingale property of $W$ and $\delta e^{-\delta s}W(s)ds = - d(e^{-\delta s}W(s)) + e^{-\delta s}dW(s)$, follows that the process $\mathbb{S}(\nu_{*})(u \wedge \sigma_{\theta_o}(t)) + \int_t^{u \wedge \sigma_{\theta_o}(t)} \delta e^{-\delta s}(\theta_o - W(s)) ds$ is an $\mathcal{F}_{u \wedge \sigma_{\theta_o}(t)}$-martingale, hence an $\mathcal{F}_{u}$-martingale. Therefore $\mathbb{S}(\nu_{*})(u \wedge \sigma_{\theta_o}(t))$ is an $\mathcal{F}_u$-submartingale with absolutely continuous compensator $A(\nu_{*})$ given by
\beq
\label{compensatorSnellBenes}
dA(\nu_{*})(s):= -\delta e^{- \delta s}\left(\theta_o-W(s)\right) ds, \quad s \in [t, \sigma_{\theta_o}(t)),
\eeq
since $W(\cdot) > \theta_o$ on $[t, \sigma_{\theta_o}(t))$, $t \geq 0$.

Finally, as the Lagrange multiplier must be flat off $\{s\geq 0: \nu_{*}(\omega,s)=\theta_o\}$ (cf.\ (\ref{dynamicCOND4}) and (\ref{dlambdaidentification})), we conclude that the Lagrange multiplier of problem (\ref{costoconvesso}) is
\beq
\label{LagrangemultBenes}
d\lambda(s)= \delta e^{- \delta s}\left[\theta_o-W(s)\right]\mathds{1}_{\{W(\cdot) -c > \theta_o\}}(s)\,ds,
\eeq
which is, as expected, coincides with the opposite of the optional measure $dA(\nu_{*})(t)$ (cf.\ (\ref{compensatorSnellBenes})).


\subsection{Constant Finite Fuel and Operating Profit of Cobb-Douglas Type}
\label{ExCD}

We consider the finite fuel version of the profit maximization problem, net of investment costs, of Riedel and Su \cite{RiedelSu} and we take the economic shock process $X(t)$ to be a Geometric Brownian motion 
\beq
\label{geoBM}
X(t) = x_0 e^{(\mu - \frac{1}{2}\sigma^2) t + \sigma W(t)} \qquad \mbox{with}\quad x_0 > 0.
\eeq
That is, we study
\beq
\label{Cobb1}
\sup_{\nu \in \mathcal{S}_{\theta_o}}\mathcal{J}(\nu):= \sup_{\nu \in \mathcal{S}_{\theta_o}}\mathbb{E}\bigg\{\,\int_0^{\infty}e^{-\delta s}R\left(X(s),\nu(s)\right) ds - \int_{0}^{\infty}e^{-\delta s} d\nu(s) \bigg\},
\eeq
where the controls satisfy $0 \leq \nu(t) \leq \theta_o$ $\mathbb{P}$-a.s., for all $t\geq 0$, with $\theta_o$ positive constant.
The firm's operating profit function is of the Cobb-Douglas type and depends on the economic shock $x$ and the investment policy $y$; i.e., $R\left(x,y\right) = \frac{1}{1 - \alpha} x^{\alpha}y^{1 - \alpha}$ with $0 < \alpha < 1$. As pointed out in \cite{RiedelSu} this construction is consistent with a competitive firm which produces at decreasing returns to scale or with a monopolist firm facing a constant elasticity demand function and constant returns to scale production.
Notice that optimization problem (\ref{Cobb1}) in the case of infinite fuel has also been studied in a diffusive setting in \cite{Kobila} by a dynamic programming approach.

It is known (cf.\ Bank \cite{Bank} and Riedel and Su \cite{RiedelSu}) that the unique optimal solution for problem (\ref{Cobb1}) is given by
\beq
\label{optCobb}
\nu_{*}(t) = \sup_{0 \leq s < t}\left(l(s) \wedge \theta_o \right) \vee \nu(0),
\eeq
where the optional process $l(t)$ uniquely solves the stochastic backward equation (cf.\ Bank and El Karoui \cite{BankElKaroui})
\beq
\label{backCobb}
\mathbb{E}\bigg\{\,\int_{\tau}^{\infty}e^{-\delta s} X^{\alpha}(s)\big(\sup_{\tau \leq u < s}l(u)\big)^{- \alpha} ds\,\Big | \mathcal{F}_{\tau}\bigg\} = e^{-\delta \tau},\,\,\,\,\,\forall \tau \in \mathcal{T}.
\eeq
As shown in Riedel and Su \cite{RiedelSu}, Proposition $7.1$, when the shock process is of exponential Levy type, i.e. $X(t) = x_0 e^{Y(t)}$, with $Y(t)$ a Levy process such that $Y(0)=0$, then the solution of (\ref{backCobb}) is given by the base capacity
\beq
\label{baseCobb}
l(t) = k X(t),
\eeq
where
$k= (\frac{1}{\delta }\mathbb{E}\{e^{\alpha \underline{Y}(\tau(\delta))}\})^{\frac{1}{\alpha}},$
$\underline{Y}(t): = \inf_{0 \leq u \leq t}Y(u)$ and $\tau(\delta)$ is an independent exponentially distributed time with parameter $\delta$.

From (\ref{Cobb1}) we have
\beq
\label{gradientCobb}
\nabla_{y}\mathcal{J}(\nu)(t) = \mathbb{E}\bigg\{\,\int_t^{\infty}e^{-\delta s} X^{\alpha}(s)\nu^{- \alpha}(s) ds\,\Big | \mathcal{F}_t \bigg\} - e^{-\delta t}.
\eeq
Following Bank \cite{Bank}, proof of Theorem 3.1, (see also the first part of the proof of our Theorem \ref{dynamictheorem} above) we know that the Snell envelope of supergradient (\ref{gradientCobb}) evaluated at the optimal control policy (\ref{optCobb}) is 
\begin{eqnarray}
\label{SnellCobb2}
\mathbb{S}(\nu_{*})(t) &\hspace{-0.25cm} = \hspace{-0.25cm}& \mathbb{E}\bigg\{\,\int_{\rho_{\theta_o}(t)}^{\infty}e^{-\delta s} \Big[X^{\alpha}(s)\Big(({\theta_o})^{- \alpha} - (\sup_{t \leq u < s}kX(u))^{- \alpha}\Big) \Big] \,ds\,\Big | \mathcal{F}_t \bigg \} \nonumber \\
&\hspace{-0.25cm}  = \hspace{-0.25cm}& {(\theta_o)}^{- \alpha}\mathbb{E}\bigg\{\,\int_{\rho_{\theta_o}(t)}^{\infty}e^{-\delta s} X^{\alpha}(s)\,ds\,\Big | \mathcal{F}_t \bigg\} - \mathbb{E}\Big\{ e^{-\delta\rho_{\theta_o}(t)}\,\Big| \mathcal{F}_t \Big\}, \nonumber
\end{eqnarray}
with
\beq
\label{tauCobb}
\rho_{\theta_o}(t):= \inf \{ s \geq t :  kX(s) > \theta_o \},
\eeq
and where we have used (\ref{backCobb}) to obtain the second equality. 
\begin{Lemma}
\label{SnellCobb3}
Assume $\delta > \mu + \sigma^2$. Then for every $t \geq 0$, one has
\beq
\label{newCD}
\mathbb{E}\bigg\{\,\int_{\rho_{\theta_o}(t)}^{\infty}e^{-\delta s} X^{\alpha}(s)\,ds\,\Big | \mathcal{F}_t \bigg\} = \frac{1}{(\delta - \mu\alpha) + \frac{1}{2}\sigma^2\alpha(1 - \alpha)}\mathbb{E}\Big\{e^{-\delta \rho_{\theta_o}(t)}X^{\alpha}(\rho_{\theta_o}(t))\,\Big| \mathcal{F}_t  \Big\}.
\eeq
\end{Lemma}
\begin{proof}
The proof follows from the Markov property and the Laplace transform of a Gaussian process.
Independence of Brownian increments, together with $W(u + \rho_{\theta_o}(t)) - W(\rho_{\theta_o}(t)) \sim W(u)$, allow us to write
\begin{eqnarray*}
\label{provaSnellCobb3}
\lefteqn{\mathbb{E}\bigg\{\,\int_{\rho_{\theta_o}(t)}^{\infty}e^{-\delta s} X^{\alpha}(s)\,ds\,\Big | \mathcal{F}_t \bigg\}
  =  \mathbb{E}\bigg\{\,\mathbb{E}\bigg\{\,\int_{\rho_{\theta_o}(t)}^{\infty}e^{-\delta s} X^{\alpha}(s)\,ds \,\Big| \mathcal{F}_{\rho_{\theta_o}(t)} \bigg\}\,\Big| \mathcal{F}_t \bigg\}} \nonumber \\  
& & =\mathbb{E}\bigg\{\,e^{-\delta \rho_{\theta_o}(t)} X^{\alpha}(\rho_{\theta_o}(t))
\mathbb{E}\bigg\{\,\int_{0}^{\infty}e^{-\delta u } e^{\alpha(\mu - \frac{1}{2}\sigma^2)u + \alpha\sigma (W(u +\rho_{\theta_o}(t)) - W(\rho_{\theta_o}(t)))} \,ds \bigg\}\,\Big| \mathcal{F}_t \bigg\} \nonumber \\
& & =\mathbb{E}\bigg\{\,e^{-\delta \rho_{\theta_o}(t)} X^{\alpha}(\rho_{\theta_o}(t))\mathbb{E}\bigg\{\,\int_{0}^{\infty}e^{-\delta u} e^{\alpha(\mu - \frac{1}{2}\sigma^2)u + \alpha\sigma W(u)} \,ds \bigg\}\,\Big| \mathcal{F}_t \bigg \} \\
& & =\mathbb{E}\bigg\{\,e^{-\delta \rho_{\theta_o}(t)} X^{\alpha}(\rho_{\theta_o}(t)) \int_{0}^{\infty}e^{- (\delta - \mu \alpha) u - \frac{1}{2}\sigma^2\alpha(1 - \alpha) u} du \,\Big| \mathcal{F}_t \bigg\}. \nonumber 
\end{eqnarray*}
Notice that $(\delta - \mu\alpha) + \frac{1}{2}\sigma^2\alpha(1 - \alpha) >0$ by the assumption, hence (\ref{newCD}) follows.
\end{proof}
Therefore 
\beq
\label{SnellCobb4}
\mathbb{S}(\nu_{*})(t)= \frac{({\theta_o})^{-\alpha}}{(\delta - \mu\alpha) + \frac{1}{2}\sigma^2\alpha(1 - \alpha)}\mathbb{E}\Big\{e^{-\delta \rho_{\theta_o}(t)}X^{\alpha}(\rho_{\theta_o}(t))\,\Big| \mathcal{F}_t \Big\} - \mathbb{E}\Big\{e^{-\delta \rho_{\theta_o}(t)}\,\Big| \mathcal{F}_t\Big\},
\eeq
by Lemma \ref{SnellCobb3} and (\ref{SnellCobb2}). 
Arguments similar to those used in Subsection \ref{ExBenes} allow us to identify the compensator part of its Doob-Meyer decomposition of $\mathbb{S}(\nu_{*})(t)$ as the Lagrange multiplier of problem (\ref{Cobb1}).
In fact, we have that $\mathbb{S}(\nu_{*})$ is an $\mathcal{F}_u$-martingale until $l$ is below the finite fuel $\theta_o$, since $\rho_{\theta_o}(u_1 \wedge \rho_{\theta_o}(t))= \rho_{\theta_o}(u_2 \wedge \rho_{\theta_o}(t))$ for all $u_1$, $u_2 \in [t, \infty)$. 
Then, if
$$\sigma_{\theta_o}(t):=\inf\{s > t: kX(s)\leq \theta_o \},$$ we show that $\mathbb{S}(\nu_{*})$ is an $\mathcal{F}_u$-supermartingale when the base capacity $l$ is above the finite fuel $\theta_o$; that is, $\{\mathbb{S}(\nu_{*})(u \wedge \sigma_{\theta_o}(t)\}_{u \geq t}$ is an $\mathcal{F}_u$-supermartingale, for any arbitrary but fixed $t \geq 0$. It suffices to prove that it is an $\mathcal{F}_{u \wedge \sigma_{\theta_o}(t)}$-supermartingale (cf.\ Revuz and Yor \cite{RevuzYor}, Corollary $3.6$). 
In fact, from $\rho_{\theta_o}(u \wedge \sigma_{\theta_0}(t)) = u \wedge \sigma_{\theta_o}(t)$ for all $u \geq t$ and $d(e^{-\delta s}X^{\alpha}(s)) = - \delta e^{-\delta s}X^{\alpha}(s)ds + e^{-\delta s}dX^{\alpha}(s)$ follows that the process $\mathbb{S}(\nu_{*})(u \wedge \sigma_{\theta_o}(t)) + \int_t^{u \wedge \sigma_{\theta_o}(t)} e^{-\delta s}(X^{\alpha}(s) (\theta_o)^{-\alpha} - \delta)ds$ is an $\mathcal{F}_{u \wedge \sigma_{\theta_o}(t)}$-martingale, hence an $\mathcal{F}_{u}$-martingale. 
On the other hand, we have $X^{\alpha}(s)(\theta_o)^{-\alpha} > {k}^{-\alpha}$ for $s \in [t, u \wedge \sigma_{\theta_o}(t))$ since $k= (\frac{1}{\delta }\mathbb{E}\{e^{\alpha \underline{Y}(\tau(\delta))}\})^{\frac{1}{\alpha}}$, $\underline{Y}(u): = \inf_{0 \leq s \leq u}\left[\left(\mu - \frac{1}{2}\sigma^2\right)s + \sigma W(s)\right]$, and $\mathbb{E}\{e^{\alpha \underline{Y}(\tau(\delta))}\}=\beta_{-}(\beta_{-} - \alpha)^{-1} < 1$ with $\beta_{-}$ the negative root of $\frac{1}{2}\sigma^2 x^2 + \left(\mu - \frac{1}{2}\sigma^2\right)x - \delta =0$ and $\alpha > 0$. Then $X^{\alpha}(s)({\theta_o})^{-\alpha} > \delta$ for all $s \in [t, u \wedge \sigma_{\theta_o}(t))$ and therefore, $\mathbb{S}(\nu_{*})(u \wedge \sigma_{\theta_o}(t))$ is an $\mathcal{F}_u$-supermartingale with absolutely continuous compensator $A(\nu_{*})$ given by
\beq
\label{Cobbcompensator}
dA(\nu_{*})(s):= e^{-\delta s}\left( X^{\alpha}(s)({\theta_o})^{-\alpha} - \delta \right)ds, \quad s \in [t,\sigma_{\theta_o}(t)),\,\,t \geq 0.
\eeq

Finally, as the Lagrange multiplier optional measure $d\lambda$ has support inside $\{t \geq 0: \nu_{*}(\omega,t)=\theta_o\}$, we conclude that for problem (\ref{Cobb1}) $d\lambda$ must be (cf.\ also (\ref{dlambdaidentification}))
\beq
\label{Cobbcompensator2}
d\lambda(t) = e^{- \delta t}\left( X^{\alpha}(t)({\theta_o})^{-\alpha} - \delta \right)\mathds{1}_{\{kX(\cdot) > \theta_o\}}(t)\,dt;
\eeq
that is, $d\lambda(t)$ coincides with the random measure $dA(\nu_{*})(t)$ (cf.\ (\ref{Cobbcompensator})).


\section{Solving a New Model with $N$ Firms and Cobb-Douglas Profits}
\label{CobbDouglasSocialPlannercase}

An extension of the model in Subsection \ref{ExCD} to the multivariate case with stochastic, time-dependent finite fuel is completely solved in this Section.
For a market with $T=+\infty$ and $N$ firms endowed with operating profit functions of Cobb-Douglas type, i.e.\ $R^{(i)}\left(x,y\right)=\frac{x^{\alpha_i}\,y^{1-\alpha_i}}{1-\alpha_i}$ with $\alpha_i \in (0,1)$, $i=1,2,...,N$, we find the solution of the Social Planner optimal investment problem (\ref{problemaEuropa}) and the explicit form of the Lagrange multiplier optional measure $d\lambda$. Even in this case we have that $d\lambda$ has a density with respect to the Lebesgue measure.

Assume the economic shock process $X(t)$ of the form $X(t) = \exp{\{bt + \sigma W(t)\}}$ for some one-dimensional standard Brownian motion $W(t)$ and $b,\sigma \in \mathbb{R}$. Then (cf.\ also Riedel and Su \cite{RiedelSu}, Proposition $7.1$) the unique optional solution of the stochastic backward equation
\beq
\label{BackCD2firms}
\mathbb{E}\bigg\{\,\int_{\tau}^{\infty} e^{- \delta s} R^{(i)}_{y}(X(s),\sup_{\tau \leq u < s}l^{(i)}(u))\, ds\,\Big | \mathcal{F}_{\tau}\bigg\} = e^{-\delta \tau}, \qquad \forall \tau \in \mathcal{T},
\eeq
is
\beq
\label{twoCD2}
l^{(i)}(t) = k_i X(t),\qquad i=1,2,...,N,
\eeq
with
$$k_i = \left(\mathbb{E}\bigg\{ \int_0^{+\infty} e^{-\delta t} e^{\alpha_i \inf_{0 \leq u < t} b u + \sigma W(u)} dt \bigg\}\right)^{\frac{1}{\alpha_i}} = \left[\frac{1}{\delta}\Big(\frac{\gamma_{-}}{\gamma_{-} - \alpha_i}\Big)\right]^{\frac{1}{\alpha_i}},\quad i=1,2,...,N,$$
where $\gamma_{-}$ is the negative root of $\frac{1}{2}\sigma^2 x^2 + b x - \delta = 0$ (cf.\ Bertoin \cite{Bertoin}, Chapter VII).

Define the optional process 
\beq
\label{beta}
\beta_i(t) := \frac{l^{(i)}(t)}{\sum_{j=1}^N l^{(j)}(t)}.
\eeq
Here $\beta_i(t)$ may be thought as the fraction of desirable investment of the $i$-th firm at time $t$.
By (\ref{twoCD2}), for $t \geq 0$ and $i=1,2,...,N,$ we have that $\beta_i(t)$ is constant in time; in fact
$\beta_i(t) = \frac{k_i}{\sum_{j=1}^N k_j} =: \beta_i.$
Fix $\tau \in \mathcal{T}$ and introduce the random times
\beq
\label{twotimesCD}
\left\{
\begin{array}{ll}
\sigma_1(\tau) = \inf\{s \geq \tau : \sum_{i=1}^N l^{(i)}(s) > \theta(s+)\} \\ \\
\sigma_2(\tau) = \inf\{s \geq \tau : l^{(i)}(s) > \beta_i \theta(s+),\,\,\,\forall i=1,2,...,N\}. 
\end{array}
\right.
\eeq

\begin{Lemma}
\label{tausigma}
For all $\tau \in \mathcal{T}$ we have $\sigma_1(\tau) = \sigma_2(\tau)$ $\mathbb{P}$-almost surely.
\end{Lemma}
\begin{proof}
Notice that (\ref{twoCD2}) implies $\sigma_1(\tau)=\inf\{s \geq \tau:  X(s) > \frac{\theta(s+)}{\sum_{j=1}^N k_j}\} = \inf\{s \geq \tau:  k_i X(s) > \beta_i \theta(s+),\,\,\,\forall i=1,2,...,N\}=\sigma_2(\tau)$.
\end{proof}

\begin{Remark}
\label{bothinvest}
If $\tau \in \mathcal{T}$ is a time of investment for all firms (that is, $\tau$ is a point of increase for $\nu^{(i)}_{*}(\cdot)$ and therefore $d\nu^{(i)}_{*}(\tau)>0$ for all $i$), then the second Kuhn-Tucker condition in (\ref{KT1}) guarantees that
\begin{equation*}
\label{equalityCD}
\mathbb{E}\bigg\{\,\int_{\tau}^{+\infty} e^{-\delta s}R^{(i)}_{y}(X(s),\nu^{(i)}_{*}(s))\,ds\,\Big|\,\mathcal{F}_{\tau}\,\bigg\} = \mathbb{E}\bigg\{\,\int_{\tau}^{+\infty} e^{-\delta s}R^{(j)}_{y}(X(s),\nu^{(j)}_{*}(s))\,ds\,\Big|\,\mathcal{F}_{\tau}\,\bigg\}.
\end{equation*}
\end{Remark}

\begin{Theorem}
\label{dynamictheoremCD}
The process $\underline{\nu}_{*}$ with components
\beq
\label{twooptimalCD}
\nu^{(i)}_{*}(t) = \sup_{0 \leq u < t}(l^{(i)}(u) \wedge \beta_i \theta(u)) \vee y^{(i)},\qquad i=1,2,...,N,
\eeq
is optimal for problem (\ref{problemaEuropa}). Moreover, the Lagrange multiplier $d\lambda(t)$ associated to (\ref{problemaEuropa}) is absolutely continuous with respect to the Lebesgue measure.
\end{Theorem}
\begin{proof}
Let us check that $\nu^{(i)}_{*}(t)$ satisfies the first order conditions of Theorem \ref{KTEuropa}.
Obviously $\sum_{i=1}^N \nu^{(i)}_{*}(t) \leq \theta(t)$ a.s.\ for all $t \geq 0$.

The arguments of the proof are similar to those in the proof of Theorem \ref{dynamictheorem}. Take $\tau \in \mathcal{T}$ arbitrary but fixed, and define the stopping time $\rho_{\theta}(\tau)$ as in (\ref{tauthetaBank}) but with $\sum_{i=1}^N l^{(i)}$ instead of $l$; that is,
\beq
\label{tauthetaEUROPA}
\rho_{\theta}(\tau):=\inf\{s \geq \tau: \sum_{i=1}^N l^{(i)}(s) > \theta(s+)\}.
\eeq
Notice that $\rho_{\theta}(\tau)$ is a time of increase for $ \sup_{\tau \leq u < s}l^{(i)}(u)$, $s > \rho_{\theta}(\tau)$, $i=1,2,...,N$, thanks to Lemma \ref{tausigma}, and also that 
$$\nu^{(i)}_{*}(s) \geq \sup_{\tau \leq u < s}l^{(i)}(u) \qquad \mbox{for}\,\,s \in (\tau,\rho_{\theta}(\tau)],$$
with equality if and only if $\tau$ is a time of investment for firm $i$.

Fix $i=1,2,...,N,$ and consider $\mathbb{E}\{\,\int_{\tau}^{\infty} e^{-\delta s} R^{(i)}_{y}(X(s),\nu_{*}^{(i)}(s)) ds\, | \mathcal{F}_{\tau}\}.$ Split the integral into two integrals $\int_{\tau}^{\rho_{\theta}(\tau)}$ and $\int_{\rho_{\theta}(\tau)}^{\infty}$.
Since $\rho_{\theta}(\tau)$ is a time of increase for every $\nu_{*}^{(i)}$, Remark (\ref{bothinvest}) holds and we may write
\begin{eqnarray}
\label{checkK12CDb}
\lefteqn{\mathbb{E}\bigg\{\,\int_{\tau}^{\infty} e^{-\delta s} R^{(i)}_{y}(X(s),\nu_{*}^{(i)}(s)) ds\, \Big| \mathcal{F}_{\tau}\bigg\} 
= \mathbb{E}\bigg\{\,\int_{\tau}^{\rho_{\theta}(\tau)} e^{-\delta s} R^{(i)}_{y}(X(s),\nu_{*}^{(i)}(s)) ds\,\Big | \mathcal{F}_{\tau}\bigg\}}   \\
& &  +\,\mathbb{E}\bigg\{\,\int_{\rho_{\theta}(\tau)}^{\infty} e^{-\delta s} \beta_i R^{(i)}_{y}(X(s),\nu_{*}^{(i)}(s)) ds\, \Big| \mathcal{F}_{\tau}\bigg\} 
+ \mathbb{E}\bigg\{\,\int_{\rho_{\theta}(\tau)}^{\infty} e^{-\delta s} \sum_{j \neq i} \beta_j R^{(j)}_{y}(X(s),\nu_{*}^{(j)}(s)) ds\,\Big | \mathcal{F}_{\tau}\bigg \} \nonumber
\end{eqnarray}
since $\mathcal{F}_{\tau} \subseteq \mathcal{F}_{\rho_{\theta}(\tau)}$.
Now, as in the proof of Theorem \ref{dynamictheorem}, we use the definition of $\rho_{\theta}(\tau)$ and the backward equation (\ref{BackCD2firms}) corresponding to $l^{(i)}$ to write
\begin{eqnarray}
\label{checkK12CDc}
\lefteqn{\mathbb{E}\bigg\{\,\int_{\tau}^{\infty} e^{-\delta s} R^{(i)}_{y}(X(s),\nu_{*}^{(i)}(s)) ds\, \Big| \mathcal{F}_{\tau}\bigg \} } \\
& & \leq   e^{-\delta \tau} + \mathbb{E}\bigg\{\,\int_{\rho_{\theta}(\tau)}^{\infty} e^{-\delta s} \Big[\sum_{i=1}^N \beta_i R^{(i)}_{y}(X(s),\sup_{0 \leq u < s}(l^{(i)}(s) \wedge \beta_i\theta(s))) - \delta \Big] ds\, \Big| \mathcal{F}_{\tau}\bigg\}, \nonumber
\end{eqnarray}
with equality if and only if $d\nu^{(i)}_{*}(\tau) > 0$.
By employing arguments as those in the proof of Bank \cite{Bank}, Theorem $3.1$, it is not hard to see that the last term in the right-hand side of (\ref{checkK12CDc}) does coincide with the Snell envelope $\mathbb{S}_{\beta}(\underline{\nu}_{*})$, $\underline{\nu}_{*}:=(\nu_{*}^{(1)}, \dots, \nu_{*}^{(N)})$, of the optional process $\mathbb{E}\{\,\int_{t}^{\infty} e^{-\delta s}\sum_{i=1}^N \beta_i R^{(i)}_{y}(X(s),\nu_{*}^{(i)}(s)) ds\,| \mathcal{F}_{t}\} - e^{-\delta t}$. Thanks to Assumption \ref{Assumptionpi}, $\mathbb{S}_{\beta}(\underline{\nu}_{*})$ is a supermartingale of class (D). If we now set (as in the proof of Theorem \ref{dynamictheorem}) $d\lambda$ equal to the unique predictable (hence optional), nondecreasing compensator of $\mathbb{S}_{\beta}(\underline{\nu}_{*})$, we obtain optimality of $\nu_{*}^{(i)}$ as in (\ref{twooptimalCD}).

Next, following again the same line of the proof of Theorem \ref{dynamictheorem}, we may also conclude that the Lagrange multiplier for the $N$-firms Social Planner problem is
$$d\lambda(t):=e^{-\delta t} \Big[\sum_{i=1}^N \beta_i(R^{(i)}_{y}(X(t),\beta_i \theta(t)) - \delta )\Big]\mathds{1}_{\{\sum_{i=1}^N l^{(i)}(\cdot) > \theta(\cdot +)\}}(t)\,dt,$$
which is an absolutely continuous, optional measure, nonnegative by Lemma \ref{piprimoer}.
\end{proof}

\begin{Remark}
For general operating profit functions satisfying Assumption \ref{Assumptionpi}, we expect the solution of the Social Planner problem (\ref{problemaEuropa}) to be
\begin{equation*}
\label{feelinggenerale}
\nu^{(i)}_{*}(t) = \sup_{0 \leq u < t}(l^{(i)}(u) \wedge \beta_i(u) \theta(u)) \vee y^{(i)},\qquad i=1,2,...,N, 
\end{equation*}
with $\beta_i(t):=\frac{l^{(i)}(t)}{\sum_{j=1}^N l^{(j)}(t)}$.
\end{Remark}

\bigskip

\textbf{Acknowledgments.} The authors thankfully acknowledge two anonymous referees for their pertinent and useful comments. The second author would also like to thank Jan Henrik Steg for the constructive discussions.



\begin{thebibliography}{199}

\bibitem{BankRiedel1}P.\ BANK and F.\ RIEDEL, \textsl{Optimal Consumption Choice with Intertemporal Substitution}, Ann.\ Appl.\ Probab., $11$ $(2001)$, pp.\ 750--788.

\bibitem{BankRiedel2}P.\ BANK and F.\ RIEDEL, \textsl{Optimal Dynamic Choice of Durable and Perishable Goods}, Discussion Paper $28/2003$ of the Bonn Graduate School of Economics (December $2003$).

\bibitem{BankElKaroui}P.\ BANK and N.\ EL KAROUI, \textsl{A Stochastic Representation Theorem with Applications to Optimization and Obstacle Problems}, Ann.~Probab., $32$ $(2004)$, pp.\ 1030--1067.

\bibitem{Bank}P.\ BANK, \textsl{Optimal Control under a Dynamic Fuel Constraint}, SIAM J.~Control Optim., $44$ $(2005)$, pp.\ 1529--1541.

\bibitem{BankKuchler}P.\ BANK and C.\ K\"UCHLER, \textsl{On Gittins' Index Theorem in Continuous Time}, Stochastic Process.~Appl., $117$ $(2007)$, pp.\ 1357--1371.

\bibitem{Bather2}J.A.\ BATHER and H.\ CHERNOFF, \textsl{Sequential Decisions in the Control of a Spaceship (Finite Fuel)}, J.\ Appl.\ Probab., $4$ $(1967)$, pp.\ 584--604.

\bibitem{Benes} V.L.\ BENE\v S, L.A.\ SHEPP and H.S.\ WITSENHAUSEN, \textsl{Some Solvable Stochastic Control Problems}, Stoch.~Stoch.~Rep., $4(1)$ $(1980)$, pp.\ 39--83.

\bibitem{Bertoin}J.\ BERTOIN, \textsl{Levy Processes}, Cambridge University Press $1996$.

\bibitem{ChowMenaldi} P.L.\ CHOW, J.L.\ MENALDI and M.\ ROBIN, \textsl{Additive Control of Stochastic Linear Systems with Finite Horizon}, SIAM J.~Control Optim., $23(6)$ $(1985)$, pp.\ 858--899.

\bibitem{DixitPindyck}A.K.\ DIXIT and R.S.\ PINDYCK, \textsl{Investment under Uncertainty}, Princeton University Press, Princeton $1994$.

\bibitem{KaratzasElKarouiSkorohod}N.\ EL KAROUI and I.\ KARATZAS, \textsl{A New Approach to the Skorohod Problem and its Applications}, Stoch.~Stoch.~Rep., $34$ $(1991)$, pp.\ 57--82.

\bibitem{tesimia}G.\ FERRARI, \textsl{On Stochastic Irreversible Investment Problems in Continuous Time: a New Approach Based on First Order Conditions}, Ph.D.\ dissertation, University of Rome `La Sapienza', $2011$.

\bibitem{Jacod}J.\ JACOD, \textsl{Calcul Stochastique et Problèmes de Martingales}, no.\ $714$ in Lecture Notes in Mathematics, Springer $1979$.

\bibitem{Kabanov}Y.\ KABANOV, \textsl{Hedging and Liquidation under Transaction Costs in Currency Markets}, Finance Stoch., $3$ $(1999)$, pp.\ 237--248.

\bibitem{Karatzas81}I.\ KARATZAS, \textsl{The Monotone Follower Problem in Stochastic Decision Theory}, Appl.~Math.~Optim., $7$ $(1981)$, pp.\ 175--189.

\bibitem{KaratzasShreve84}I.\ KARATZAS and S.E.\ SHREVE, \textsl{Connections between Optimal Stopping and Singular Stochastic Control I. Monotone Follower Problems}, SIAM J.~Control Optim., $22$ $(1984)$, pp.\ 856--877.

\bibitem{Karatzas85}I.\ KARATZAS, \textsl{Probabilistic Aspects of Finite-Fuel Stochastic Control}, Proc.\ Natl.\ Acad.\ Sci.\ USA, $82$ $(1985)$, pp.\ 5579--5581.

\bibitem{Kobila}T.\O.\ KOBILA, \textsl{A Class of Solvable Stochastic Investment Problems Involving Singular Controls}, Stoch.~Stoch.~Rep., $43$ $(1993)$, pp.\ 29--63.

\bibitem{RevuzYor}D.\ REVUZ and M.\ YOR, \textsl{Continuous Martingales and Brownian Motion}, Springer-Verlag, Berlin $1999$.

\bibitem{RiedelSu}F.\ RIEDEL and X.\ SU, \textsl{On Irreversible Investment}, Finance Stoch., $15(4)$ $(2011)$, pp.\ 607--633.

\end{thebibliography}
\end{document}